\documentclass{amsart}

\usepackage{amsmath,amsthm,indentfirst}
\usepackage{amssymb}
\usepackage{amsfonts}
\usepackage{amscd}
\usepackage[latin1]{inputenc}
\DeclareMathAlphabet{\mathpzc}{OT1}{pzc}{m}{it}
\theoremstyle{plain}
\newtheorem*{maintheorem*}{Main Theorem}
\newtheorem*{thm*}{Theorem}
\newtheorem*{thma*}{Theorem A}
\newtheorem*{thmaa*}{Theorem A'}
\newtheorem*{thmb*}{Theorem B}
\newtheorem*{thmo*}{Theorem 1.1}
\newtheorem*{thmc*}{Theorem C}
\newtheorem*{thmd*}{Theorem D}
\newtheorem*{thmf*}{Theorem 4.1}
\newtheorem*{remark*}{Remark}
\newtheorem*{conjecture*}{Conjecture}
\newtheorem*{prop*}{Proposition}
\newtheorem*{lem*}{Basic Lemma}
\newtheorem{thm}{Theorem}[section]
\newtheorem{cor}[thm]{Corollary}
\newtheorem{lem}[thm]{Lemma}
\newtheorem{prop}[thm]{Proposition}

\theoremstyle{definition}

\newtheorem*{proofc*}{Proof of Theorem C}

\newtheorem{definition}[thm]{Definition}

\newtheorem{remark}[thm]{Remark}

\def\bbz{\mathbb{Z}}
\def\bbq{\mathbb{Q}}

\def\bbr{\mathbb{R}}

\def\bbn{\mathbb{N}}

\def\qfrak{\mathfrak{q}}
\def\pfrak{\mathfrak{p}}

\def\afrak{\mathfrak{A}}
\def\bfrak{\mathfrak{B}}
\def\vare{\varepsilon}

\def\Lxi{\Lambda_{\xi}}
\def\Lqxi{\Lambda_{\mathfrak{q}\xi}}
\def\itl{I_t^L(\delta)}
\def\itlone{I_t^{\pi_1(L)}(\delta)}
\def\itltwo{I_t^{\pi_2(L)}(\delta)}
\def\itli{I_t^{\pi_i(L)}(\delta)}

\def\rb{\rangle_B}
\def\rqt{\rangle_{Q_T}}
\def\hrn{H\ltimes\bbr^n}
\def\qtd{\mathcal{Q}_t(\delta,\vare,\tau_2)}
\def\qtv{\mathcal{Q}_t(\delta,\vare)}

\def\atld{\afrak_t^L(\delta,\vare)}

\def\SL{\rm{SL}}

\def\h{\hspace{1mm}}
\def\hh{\hspace{.5mm}}

\title[Inhomogeneous quadratic forms]{\sc Quantitative Version of the Oppenheim Conjecture for Inhomogeneous Quadratic Forms}
\author{G.~A.~ Margulis \& A.~Mohammadi}
\date{}

\thanks {G.M. was partially supported by NSF grant DMS-0801195.}
\address{Mathematics Dept., Yale University, New Haven, CT}
\email{margulis@math.yale.edu}
\address{Mathematics Dept., University of Chicago, Chicago, IL}
\email{amirmo@math.uchicago.edu}

\begin{document}
\maketitle
\begin{abstract}
A quantitative version of the Oppenheim conjecture for inhomogeneous quadratic forms is proved. We also give an application to eigenvalue spacing on flat 2-tori with Aharonov-Bohm flux.
\end{abstract}


\section{Introduction}\label{sec:intro}
Let $Q$ be a nondegenerate indefinite quadratic form on $\bbr^n.$ Let $\xi\in\bbr^n$ be a vector and define the (inhomogeneous) quadratic form $Q_{\xi}$ by
\begin{equation}\label{inhomoform}Q_{\xi}(x)=Q(x+\xi)\hspace{3mm}\mbox{for all}\h\h x\in\bbr^n\end{equation}
We will refer to $Q=Q_{\bf 0}$ as the homogeneous part of $Q_{\xi}.$ We say $Q_\xi$ has signature $(p,q)$ if $Q$ does. Recall that a quadratic form $Q_\xi$ is called irrational if it is not scalar multiple of a form with rational coefficients. In other words $Q_\xi$ is irrational if either $Q$ is irrational as a homogeneous form or if $Q$ is a rational form then $\xi$ is an irrational vector.

Let $\nu$ be a continuous function on the sphere $\{v\in\bbr^n\h:\h\|v\|=1\}.$ Define $\Omega=\{v\in\bbr^n\h:\h \|v\|<\nu(v/{\|v\|})\}$ and let $T\Omega$ be the dilate of $\Omega$ by $T.$ For an indefinite quadratic form $Q$ in $n$ variables and a vector $\xi\in\bbr^n$ we let
\begin{equation}\label{countingfunction} N_{Q,\xi,\Omega}(a, b, T)=\#\h\{x\in\bbz^n\h:\h x\in T\Omega\h\h\mbox{and}\h a<Q_\xi(x)<b\}\end{equation}
If $\xi=\mathbf{0}$ we let $N_{Q,\mathbf{0},\Omega}(a, b, T)=N_{Q,\Omega}(a, b, T).$ It is easy to see that there exists a constant $\lambda_{Q,\Omega}$ such that
\begin{equation}\label{volume} \mbox{Vol}(\{x\in\bbr^n\h:\h x\in T\Omega\h\h\mbox{and}\h a<Q_\xi(x)<b\})\sim\lambda_{Q,\Omega}(b-a)T^{n-2}\end{equation}

A.~Eskin, G.~A.~Margulis and S.~Mozes in~\cite{EMM1} proved

\begin{thm}\label{emm1}{\rm{(\cite[Theorem 2.1]{EMM1})}} Let $Q$ be a quadratic form of signature $(p,q),$ with $p\geq3$ and $q\geq1.$ Suppose $Q$ is not proportional to a rational form. Then for any interval $(a,b)$
\begin{equation}\label{e:emm1}N_{Q,\Omega}(a,b,T)\sim\lambda_{Q,\Omega}(b-a)T^{n-2}\hspace{2mm}{\rm{as}}\h\h T\rightarrow\infty\end{equation}
where $n=p+q$ and $\lambda_{Q,\Omega}$ is as in~(\ref{volume}).
\end{thm}

Theorem~\ref{emm1} fails if $Q$ has signature $(2,2)$ or $(2,1).$ Indeed there are irrational forms for which along a sequence $T_j,$ $N_{Q,\Omega}(a,b,T_j)>T_j^{n-2}(\log T_j)^{1-\epsilon}.$ However these irrational forms are very well approximated by split rational forms. Let us recall the following definition from~\cite{EMM2}.

\begin{definition}\label{ewas}
A quadratic form $Q$ of signature $(2,2)$ is called {\it extremely well approximable by split forms} (EWAS) if for every $N>0$ there exists a split integral form $Q'$ and $2\leq r\in\bbr$
\begin{equation}\label{e:ewas}\left\|Q-\frac{1}{r}Q'\right\|\leq\frac{1}{r^N}\end{equation}
where $\|\h\|$ is a norm, fixed once and for all, on the space of quadratic forms in four variables.
\end{definition}

Recall from~\cite{EMM2} that if $Q$ is irrational of signature $(2,2)$ then it has at most 4 rational null subspaces. Let
\begin{equation}\label{tildecountingfunction} \tilde{N}_{Q,\Omega}(a, b, T)=\#\h\left\{x\in\bbz^n\h:\h \begin{array}{c}x\h\h\mbox{is not in a null subspace of}\h\h Q \\ x\in T\Omega\h\h\mbox{and}\h\h a<Q(x)<b\end{array}\right\}\end{equation}

A.~Eskin, G.~A.~Margulis and S.~Mozes in~\cite{EMM2} proved

\begin{thm}\label{emm2}{\rm{(\cite[Theorem 1.3]{EMM2})}}
Let $\Omega$ be as above. Let $Q$ be an irrational quadratic form of signature $(2,2)$ which is not {\rm{EWAS}}.  Then for any interval $(a,b)$
\begin{equation}\label{e:emm1}\tilde{N}_{Q,\Omega}(a,b,T)\sim\lambda_{Q,\Omega}\hh(b-a)\hh T^{2}\hspace{2mm}{\rm{as}}\h\h T\rightarrow\infty\end{equation}
where $\lambda_{Q,\Omega}$ is as in~(\ref{volume}),  and $\tilde{N}_{Q,\Omega}$ is as in~(\ref{tildecountingfunction}).
\end{thm}

This paper extends theorems~\ref{emm1} and~\ref{emm2} to the setting of inhomogeneous quadratic forms. 
Let us first state the following which provides us with asymptotically exact lower bound. This indeed is the analogue of the similar result obtained by S.~G.~Dani and G.~A.~Margulis~\cite{DM3}.

\begin{thm}\label{dmlowerbound}
Let $Q_\xi$ be an indefinite irrational quadratic form in $n\geq3$ variables. Then for any interval $(a,b)$ we have
\begin{equation}\label{e:dmlowerbound}{\lim{\rm{inf}}}\frac{1}{T^{n-2}}N_{Q,\xi,\Omega}(a,b,T)\geq\lambda_{Q,\Omega}(b-a)\hspace{3mm}\mbox{as}\h\h{T\rightarrow\infty}\end{equation}
where $\lambda_{Q,\Omega}$ is as in~(\ref{volume})
\end{thm}
\begin{proof}
This theorem is obtained from theorems~\ref{bequidistribution1} and~\ref{bequidistribution2} below with arguments as in~\cite{DM3} or~\cite[3.4,3.5]{EMM1}.
\end{proof}

We have the following

\begin{thm}\label{signaturenot(2,2)}
Let $Q$ be an indefinite quadratic form of signature $(p,q)$ where $p\geq3$ and $q\geq1.$ Let $\xi\in\bbr^n$ where $n=p+q.$ Suppose that $Q_\xi$ is an irrational form then
\begin{equation}\label{e:signaturenot(2,2)}N_{Q,\xi,\Omega}(a,b,T)\sim\lambda_{Q,\Omega}(b-a)T^{n-2}\hspace{2mm}{\rm{as}}\h\h T\rightarrow\infty\end{equation}
where $\lambda_{Q,\Omega}$ is given in~(\ref{volume}).
 \end{thm}

As in~\cite{EMM1} we also have the following uniform version of theorem~\ref{signaturenot(2,2)}. Let $\mathcal{I}(p,q)$ denote the space of inhomogeneous quadratic forms whose homogeneous parts are quadratic forms of signature $(p,q)$ and discriminant $\pm1.$

\begin{thm}\label{uniformnot(2,2)}
Let $\mathcal{D}$ be a compact subset of $\mathcal{I}(p,q)$ with $p\geq3$ and $q\geq1$ and let $n=p+q.$ Then for every interval $(a,b)$ and every $\theta>0$ there exists a finite subset $\mathcal{P}$ of $\mathcal{D}$ such that each $Q_{\xi}\in\mathcal{P}$ is a rational form and for every compact subset $F\subset\mathcal{D}\setminus\mathcal{P}$ there exists $T_0$ such that for all $Q_\xi\in F$ and $T\geq T_0$
\begin{equation}\label{e:uniformnot(2,2)}(1-\theta)\lambda_{Q,\Omega}(b-a)T^{n-2}\leq{N}_{Q,\xi,\Omega}(a,b,T)\leq(1+\theta)\lambda_{Q,\Omega}(b-a)T^{n-2}\end{equation}
where $\lambda_{Q,\Omega}$ is as in~(\ref{volume}).
\end{thm}

The proofs of the above theorems are straightforward inhomogeneous versions of the arguments and ideas developed in~\cite{DM3} and~\cite{EMM1}.

As we mentioned before theorem~\ref{signaturenot(2,2)} fails in signature $(2,2)$ and $(2,1)$. In this paper we prove an inhomogeneous version of theorem~\ref{emm2}. Indeed as in loc. cit. one needs to assume certain ``{\it Diophantine condition}" on the quadratic form. Using similar ideas we also give a partial result in the $(2,1)$ case, see theorem~\ref{(2,1)} below. We start with the following definitions.

\begin{definition}\label{d:farfromrational}
A vector $\xi=(\xi_1,\cdots,\xi_n)\in\bbr^n$ is called {\it $\kappa$-Diophantine}, if there exist $\kappa>0$ and $C=C(\xi)>0$ such that for all $0<\delta<1$ and all rational vectors $(\frac{p_1}{q_1},\cdots,\frac{p_n}{q_n})\in\bbq^n$ with $\max_i|q_i|<1/\delta$ we have
\begin{equation}\label{farfromrational}\max_i |\xi_i-\frac{p_i}{q_i}|>C\hh\delta^{\kappa}\end{equation}
\end{definition}
We say $\xi$ is Diophantine if it is $\kappa$-Diophantine for some $\kappa.$ 
The following is our Diophantine condition on $Q_\xi.$
\begin{definition}\label{diophantineforms}
The irrational inhomogeneous quadratic form $Q_\xi$ of signature $(2,2)$ is called {\it Diophantine} if either $Q$ is not EWAS or $\xi$ is Diophantine.
\end{definition}

Let $Q_\xi$ be a quadratic form of signature $(2,2).$ If $L$ is a rational $2$-dimensional null subspace of $Q$ and $\xi\in L+v_\xi$ for some $v_\xi\in\bbz^4$ then the affine subspace $L-v_\xi$ will be called an exceptional subspace of $Q_\xi.$ We will see that if either $Q$ or $\xi$ is irrational then there are at most four subspaces $L$ for which the above can hold. Let
\begin{equation}\label{inhcountingfunction} \tilde{N}_{Q,\xi,\Omega}(a, b, T)=\#\h\left\{x\in\bbz^n\h:\h \begin{array}{c}x\h\mbox{is not in an exceptional subspace of}\h Q_\xi \\ x\in T\Omega\h\h\mbox{and}\h\h a<Q_\xi(x)<b\end{array}\right\}\end{equation}
The following is analogue of theorem~\ref{emm2} in the inhomogeneous setting.

\begin{thm}\label{(2,2)}
Let $Q_\xi$ be an inhomogeneous quadratic form of signature $(2,2).$ Assume that $Q_{\xi}$ is Diophantine.  Then for any interval $(a,b)$
\begin{equation}\label{e:(2,2)}\tilde{N}_{Q,\xi,\Omega}(a,b,T)\sim\lambda_{Q,\Omega}(b-a)T^2\hspace{2mm}{\rm{as}}\h\h T\rightarrow\infty\end{equation}
where $\lambda_{Q,\Omega}$ is as in~(\ref{volume}),  and $\tilde{N}_{Q,\xi,\Omega}$ is defined in~(\ref{inhcountingfunction}).
\end{thm}

We now turn to the $(2,1)$ case. Our result in this case is more restrictive. As before we define the notion of exceptional subspaces for forms of signature $(2,1).$ These are affine subspaces $L-v_\xi$ such that $L$ is a rational $1$-dimensional null subspace of $Q$ and $\xi\in L+v_\xi$ for some $v_\xi\in\bbz^3.$ The counting function, $\tilde{N}_{Q,\xi,\Omega},$ is defined correspondingly. We have the following

\begin{thm}~\label{(2,1)}
Let $Q_\xi$ be a form of signature $(2,1).$ If
\begin{itemize}
\item[(i)] the homogeneous part $Q$ is a split rational form and
\item[(ii)] the vector $\xi$ is Diophantine.
\end{itemize}
then for any interval $(a,b)$
\begin{equation}\label{e:(2,1)}\tilde{N}_{Q,\xi,\Omega}(a,b,T)\sim\lambda_{Q,\Omega}(b-a)T\hspace{2mm}{\rm{as}}\h\h T\rightarrow\infty\end{equation}
where $\lambda_{Q,\Omega}$ is as in~(\ref{volume}),  and $\tilde{N}_{Q,\xi,\Omega}$ is defined in~(\ref{inhcountingfunction}).
\end{thm}

The proofs of theorems~\ref{(2,2)} and~\ref{(2,1)} require new ingredients combined with ideas developed in~\cite{EMM1} and~\cite{EMM2}. Indeed what is new in theorem~\ref{(2,2)} compare to~\cite{EMM2} is the case when the homogeneous part $Q$ is EWAS. An important especial case is when $Q$ is a split rational form. In this case, and also theorem~\ref{(2,1)} above, we need to study the contribution coming from null subspaces to the counting function $\tilde{N}.$ This is done in sections~\ref{sec:quasinull} and~\ref{sec;(2,1)}.

{\it Eigenvalue spacing on flat 2-tori.} It has been conjectured by Berry and Tabor~\cite{BT} that the the eigenvalues of a generic quantized completely integrable Hamiltonian follow the statistics of a Poisson point-process, i.e their consecutive spacings should be independent and identically distributed exponentially distributed. Except some numerical experiments results which support the Berry-Tabor conjecture on a rigorous level have so far only been obtained for a statistic which is easier to handle. This is the pair correlation density function.

Let the Hamiltonian be the geodesic flow for flat 2-torus. It was proved by P.~Sarnak~\cite{S} that for almost all (with respect to Lebesgue measure on the moduli space of two dimensional flat tori) two dimensional flat tori the pair correlation density function converges to the pair correlation density of a Poisson process. He uses averaging arguments to reduce pair correlation problem to a problem about spacing between the values at integers of binary quadratic forms. This is related to the quantitative Oppenheim problem in the case of signature $(2,2).$ One corollary of theorem~\ref{emm2} is that Berry-Tabor conjecture holds for pair correlation of two dimensional flat tori under certain explicit Diophantine condition.

Similarly theorem~\ref{(2,2)} has a corollary in this direction. Let $\mathfrak{h}$ be a lattice in $\bbr^2$. Also let $\alpha=(\alpha_1,\alpha_2)\in\bbr^2.$ Now the eigenvalues of the Laplacian
\begin{equation}\label{laplacian}-\Delta=-\frac{\partial^2}{\partial y^2}-\frac{\partial^2}{\partial y^2}\end{equation}
with quasi periodicity conditions
\begin{equation}\label{quasiperiod}\phi(x+v)=e^{2\pi i\langle\alpha,v\rangle}\phi(x)\h\h\mbox{for all}\h\h x\in\bbr^2\h\h\mbox{and all}\h\h v\in\mathfrak{h}\end{equation}
are of the form $4\pi^2\|w+\alpha\|^2$ where $w\in\mathfrak{h}^*$ and $\mathfrak{h}^*$ is the dual lattice to $\mathfrak{h}.$ Let
\begin{equation}\label{eigenvalues}0\leq\lambda_0<\lambda_1\leq\lambda_2\cdots\end{equation}
be these eigenvalues counted with multiplicity. By the Weyl's law we have
\begin{equation}\label{weyl}\#\{j: \lambda_j\leq T\}\sim c_{\mathfrak{h}}\hh T\end{equation}
where $c_{\mathfrak{h}}=\frac{\mbox{covol}(\mathfrak{h})}{4\pi}.$ Let $0\notin(a,b)$ and define the {\it pair correlation function}
\begin{equation}\label{paircorel}
R_{\mathfrak{h},\alpha}(a,b,T)=\frac{\#\{(j,k): \lambda_j<T,\h\lambda_k<T,\h a\leq\lambda_j-\lambda_k\leq b\}}{T}
\end{equation}

Let now $\{w_1,w_2\}$ be a basis for $\mathfrak{h}^*$ and let $\beta=(\beta_1,\beta_2)$ be so that $\beta_1w_1+\beta_2w_2=\alpha.$ Consider $B(x_1,x_2)=4\pi^2\|x_1w_1+x_2w_2+\alpha\|^2.$ Indeed the above mentioned eigenvalues are the values at integer points of the form $B_\beta.$ Let $\xi=(\beta,\beta)\in\bbr^4$ and define $Q(x_1,x_2,x_3,x_4)=B(x_1,x_2)-B(x_3,x_4).$ The pair correlation function for these eigenvalues, in the case $0\notin(a,b),$ is now asymptotically the counting function $\tilde{N}_{Q,\xi,\Omega}(a,b,T),$ with $\Omega=\{x:\h\max(B(x_1,x_2)^{1/2}, B(x_3,x_4)^{1/2})\leq1\}.$ Thus we obtain the following

\begin{cor}\label{cor;paircorel}
Let $\mathfrak{h}$ be a lattice in $\bbr^2$ and let $B$ be the quadratic form obtained as above normalized so that one of the coefficients is $1$ and let $A_1$ and $A_2$ be other coefficients. Also let vector $\alpha\in\bbr^2$ be given and define $\beta$ as above. Suppose that at least one of the following holds
\begin{itemize}
\item[(i)] The vector $\beta=(\beta_1,\beta_2)$ is Diophantine.
\item[(ii)] There exists $N,C>0$ such that for all triples of integers $(p_1,p_2,q)$ with $q\geq2,$
$$\max_{i=1,2}\left|A_i-\frac{p_i}{q}\right|>\frac{C}{q^N}$$
\end{itemize}

Then for any interval $(a,b)$ with $0\notin(a,b)$ we have
\begin{equation}\label{e;paircor}
\lim_{T\rightarrow\infty}R_{\mathfrak{h},\alpha}(a,b,T)=c_{\mathfrak{h}}^2(b-a)
\end{equation}
Hence the spectrum satisfies the Berry-Tabor conjecture for pair correlation function.
\end{cor}

In the case $\mathfrak{h}=\bbz^2$ this was proved by J.~Marklof~\cite{M}. His approach utilizes results from theory of unipotent flows combined with application of theta sums. We also use the theory of unipotent flows in our proof however our strategy to control the integral of unbounded functions over certain orbits is dynamical and rests heavily on~\cite{EMM1} and~\cite{EMM2}.

\textit{Outline of the proof.} Let $Q_\xi$ be a quadratic form of signature $(p,q)$ and let $n=p+q.$ Fix an interval $(a,b)$ and let $U\subset\bbr^n$ be a ``suitably chosen" compact set such that $a<Q(u)<b,$ for all $u\in U.$ We want to count the number of vectors $v\in\bbz^n$ with $T/2\leq\|v+\xi\|\leq T,$ such that $a<Q(v+\xi)<b.$ Note that ${\rm{SO}}(Q)$ acts transitively on the level sets of $Q$ hence there exists some $g\in{\rm{SO}}(Q)$ such that $g(v+\xi)\in U.$ Now if we let $\hat{f}(g(\bbz^n+\xi))=\#(g(\bbz^n+\xi)\cap U)$ then ${N}_{Q,\xi}(a,b,T)-{N}_{Q,\xi}(a,b,T/2)$ can be approximated by the integral $T^{n-2}\int_H\hat{f}(g(\bbz^n+\xi))dg.$ Integrals of this form are the main object of study in~\cite{EMM1} and~\cite{EMM2}. The question in hand is that of equidistribution results for unbounded functions. One obtains the lower bound by approximating $\hat f$ by compactly supported functions as it was done in~\cite{DM3}. However in order to obtain the upper bound one needs to deal with the structure at infinity of the space of lattices or in our case the space of inhomogeneous lattices.

{\bf Acknowledgments.}
We would like to thank J. Marklof for reading the first draft and many helpful comments.


\section{Passage to space of inhomogeneous lattices}\label{sec;inhomo}
As was outlined above, and is done in~\cite{EMM1} and~\cite{EMM2}, our approach is to translate the problem into a problem on homogeneous spaces and then borrow from the rich structure in there. As we are working with inhomogeneous forms the space of main interest will be $\SL_n(R)\ltimes\bbr^n/\SL_n(\bbz)\ltimes\bbz^n,$ which is naturally identified with the space of inhomogeneous unimodular lattices in $\bbr^n.$

\textit{Quadratic forms.} Let $n\geq3$ and let $n=p+q$ where $p\geq2.$ Let $\{e_1,\cdots,e_n\}$ be the standard basis for $\bbr^n.$ If $p\geq3$ let $B$ be the ``standard" form
\begin{equation}\label{standardnot(2,2)}B\left(\sum_{i=1}^{i}x_ie_i\right)=2x_1x_n+\sum_{i=2}^{p}x_i^2-\sum_{i=p+1}^{n-1}x_i^2\end{equation}
Let $H={\rm{SO}}(B)$ and $\{a_t\}$ be the one-parameter subgroup of $H$ given by $a_te_1=e^{-t}e_1,$ $a_te_i=e_i$ for $2\leq i\leq n-1$ and $a_te_n=e^{t}e_n.$ And let $K=H\cap\hat{K}$ where $\hat{K}$ is the group of orthogonal matrices with determinant 1. We let $dk$ denote the Haar measure on $K$ normalized so that $K$ is a probability space.

If $(p,q)=(2,2)$ we let
\begin{equation}\label{standard(2,2)}B(x_1,x_2,x_3,x_4)=x_1x_4-x_2x_3\end{equation}
be the standard form on $\bbr^4.$ This is the determinant on $\mbox{M}_2(\bbr),$ if we identify $\bbr^4$ with $\mbox{M}_2(\bbr).$ Note that this identification shows that $\mbox{SO}(2,2)$ is locally isomorphic to $\SL_2(\bbr)\times\SL_2(\bbr)$ with the action $v\rightarrow g_1vg_2^{-1},$ which leaves the determinant invariant. We let $H=\SL_2(\bbr)\times\SL_2(\bbr),$ $K=\mbox{SO}(2)\times\mbox{SO}(2)$ and $a_t=(b_t,b_t)$ where $b_t=\mbox{diag}(e^{-t/2},e^{t/2}).$ We let $dk$ denote the Haar measure on $K$ normalized so that $K$ is a probability space. We will often work with the standard lattice $\bbz^4$ and the form $Q$ in which case we continue to denote by $\{a_t\}$ and $K$ the corresponding one parameter and maximal compact subgroup of $\mbox{SO}(Q).$

If $(p,q)=(2,1)$ we let
\begin{equation}\label{standard(2,1)}B(x_1,x_2,x_3)=x_1x_3-x_2^2\end{equation}
be the standard form on $\bbr^3.$ This is the determinant on $\mbox{Sym}_2(\bbr),$ the space $2\times2$ symmetric matrices, if identify $\bbr^3$ with $\mbox{Sym}_2(\bbr).$ This identification shows that $\mbox{SO}(2,1)$ is locally isomorphic to $\SL_2(\bbr)$ with the action $v\rightarrow gv{}^tg,$ where ${}^tg$ is the transpose matrix. We let $H=\SL_2(\bbr).$ We let $a_t=\mbox{diag}(e^{-t/2},e^{t/2})$ and let $K=\mbox{SO}(2)$ be the maximal compact subgroup of $H.$ As before $dk$ denotes the normalized Haar measure on $K.$

Let $f$ be a continuous function with compact support on $\bbr^n$ we define the theta transform of $f$ by
\begin{equation}~\label{thetatrans}\hat{f}(\Lambda+\xi)=\sum_{v\in\Lambda+\xi}f(v)\end{equation}
where $\Lambda+\xi$ is any unimodular inhomogeneous lattice in $\bbr^n.$ Note that $\hat{f}$ is a function on the space of inhomogeneous lattices.

We fix some more notations. Let $n=p+q$ and let $G=\SL_n(\bbr)\ltimes\bbr^n.$ Let $\Gamma=\SL_n(\bbz)\ltimes\bbz^n$ which is a lattice in $G.$
We have the following, which is similar to Siegel's integral formula.

\begin{lem}\label{siegelformula}
Let $f$ and $\hat{f}$ be as above. Let $\mu$ be a probability measure on $G/\Gamma$ which is invariant under $\bbr^n.$ Then
\begin{equation}\label{e:siegelformula}\int_{G/\Gamma}\hat{f}(g)\hh d\mu(g)=\int_{\bbr^n}f(x)\hh dx\end{equation}
\end{lem}

\begin{proof}
Note that $\bbr^n$ is the unipotent radical of $G.$ Now the lemma follows from the fact that $\mu$ is $\bbr^n$-invariant and Fubini's theorem.
\end{proof}

We end this section by recalling the definition of the $\alpha$ functions defined on the space of lattices. Let $\Delta$ be a lattice in $\bbr^n.$ A subspace $L$ of $\bbr^n$ is called $\Delta$-{\it rational} if $L\cap\Delta$ is a lattice in $L.$ For a $\Delta$-rational subspace $L$ let $\frac{1}{d(L)}$ be the volume of $L/(L\cap\Delta).$ For $0\leq i\leq n$ define
\begin{equation}\label{alphai}\alpha_i(\Delta)=\sup\left\{\frac{1}{d(L)}\h:\h L\h\mbox{is a}\h\Delta\mbox{-rational subspace of dimension}\h i\right\}\end{equation}
and let $\alpha(\Delta)=\max_i\alpha_i(\Delta).$ Now if $\Delta_\xi=\Delta+\xi$ is an inhomogeneous lattice let $\alpha(\Delta_\xi)=\alpha(\Delta).$ There is a constant $c=c(f)$ depending on $f$ such that for any inhomogeneous lattice $\Delta_\xi=\Delta+\xi$ we have
\begin{equation}\label{lipschitzprinciple}\hat{f}(\Delta_\xi)<c\h\alpha(\Delta_\xi)=c\h\alpha(\Delta)\end{equation}
This is analogue of~\cite[Lemma 2]{Sc} and the proof is similar.


\section{The case of where $p\geq3$.}\label{sec;not(2,2)}
In this section we prove theorem~\ref{uniformnot(2,2)} modulo results proved in appendix~\ref{sec:equi}. As we mentioned the proof is an easy adaptation of the proof of~\ref{emm1}. We include it for the sake of completeness. Hence through out this section we assume $p\geq3$ and $q\geq1.$ Let us recall the following

\begin{thm}\label{emm1alpha}{\rm{(\cite[Theorem 3.2]{EMM1})}}
If $p\geq3$ and $q\geq1$ and $0<s<2$ then for any lattice $\Delta$ in $\bbr^n$
\begin{equation}\label{e:emm1alpha}\sup_{t>0}\int_K\alpha^s(a_tk\Delta)dk<\infty\end{equation}
The upper bound is uniform as $\Delta$ varies over compact sets in the space of lattices.
\end{thm}
Theorem~\ref{uniformnot(2,2)} is proved using the following which is a result of combining theorems~\ref{emm1alpha},~\ref{bequidistribution1} and \ref{bequidistribution2}. We have

\begin{thm}\label{unbddequinot(2,2)}{\rm{(cf. \cite[Theorem 3.5]{EMM1})}}
Suppose $p\geq3$ and $q\geq1.$ Let $f$ and $\hat{f}$ be as above. 
Let $\nu$ be any continuous function on $K.$ Then for every compact subset $\mathcal{D}$ of $G/\Gamma$ there exists finitely many points $x_1,\cdots,x_\ell\in G/\Gamma$ such that
\begin{itemize}
\item[(i)] the orbit $H\hh x_i$ is closed and has finite $H$-invariant measure, for all $i,$
\item[(ii)]  for any compact set $F\subset\mathcal{D}\setminus\bigcup_iH\hh x_i$ there exists $t_0>0$ such that for all $x\in F$ and $t>t_0$
\begin{equation}\label{ubddaverage}\left|\int_K\hat{f}(a_tkx)\hh\nu(k)\hh dk-\int_{G/\Gamma}\hat f\hh d\mu\int_K\nu\hh dk\right|\leq\vare\end{equation}
\end{itemize}
where $\mu$ is either the $G$-invariant measure on $G/\Gamma$ or $H\ltimes\bbr^n\hh x$ is closed and has $H\ltimes\bbr^n$-invariant probability measure and $\mu$ is this measure.

\end{thm}

\begin{proof}
We may as we will assume that $\phi$ is non negative. Now define
\begin{equation}\label{cusp1}A(r)=\{\Delta\in G/\Gamma\h:\h \alpha_1(\Lambda)>r\}\end{equation}
Let $g_r$ be a continuous function on $G/\Gamma$ such that $g_r(\Delta)=0$ if $\Delta\notin A(r),$ $g_r(\Delta)=1$ for all $\Delta\in A(r+1)$ and $0\leq g_r(\Delta)\leq1$ if $r\leq\alpha_1(\Delta)\leq r+1.$ We have $\hat f=(\hat f-\hat f g_r)+\hat f g_r.$ Note that $\hat f-\hat f g_r$ is a continuous function with compact support on $G/\Gamma$.

Note that $H^0\ltimes\bbr^n$ is a maximal connected subgroup of $G.$ Hence for every $\delta>0$ there exists $r_0$ such that if $H\ltimes\bbr^ny$ is a closed orbit of $H\ltimes\bbr^n$ in $G/\Gamma$ with an $H\ltimes\bbr^n$-invariant probability measure $\sigma$ then $\sigma(A(r)\cap H\ltimes\bbr^ny)<\delta$ for any $r>r_0.$ Hence for $r$ sufficiently large we get
\begin{equation}\label{noscapeofmass}\left|\int_{G/\Gamma}\hat f\hh d\mu-\int_{G/\Gamma}(\hat f-\hat f g_r)d\mu\right|<\vare/3\end{equation}
where $\mu$ is as in the statement of the theorem~\ref{unbddequinot(2,2)}.

Recall that $g_r(y)=0$ if $\alpha_r\leq r.$ Let now $\beta=2-s.$ There exists a constant $B_1$ depending on $f$ such that we have
\begin{equation}\label{ubddcontrol1}\left|\int_K(\hat f g_r)(a_tkx)\hh\nu(k)\hh dk\right|\leq B_1r^{-\frac{\beta}{2}}\int_k\alpha(a_tkx)^{2-\frac{\beta}{2}}|\nu(k)|\hh dk\end{equation}
Hence if we apply theorem~\ref{emm1alpha} then there is a constant $B$ depending on $B_1$
\begin{equation}\label{ubddcontrol2}\int_K(\hat f g_r)(a_tkx)\hh\nu(k)\hh dk\leq B(\sup_{k\in K}|\nu(k)|)r^{-\frac{\beta}{2}}\end{equation}
for all $x\in\mathcal{D}.$


Now choose $r>r_0$ sufficiently large so that $B(\sup_{k\in K}|\nu(k)|)r^{-\frac{\beta}{2}}<\vare/3.$ First note that applying theorem~\ref{bequidistribution2} with the bounded continuous function $\hat f-\hat f g_r$ there are points $y_1,\cdots,y_k$ such that $H\ltimes\bbr^ny_i$'s are closed and have finite $H\ltimes\bbr^n$-invariant measure such that~(\ref{ubddaverage}) holds for any $x\in\mathcal{D}\setminus\bigcup_{i=1}^kH\ltimes\bbr^ny_i$ for $\hat f-\hat f g_r$ instead of $\hat f$ and with $\vare/3$ instead of $\vare.$ Now if we apply theorem~\ref{bequidistribution1} to $H\ltimes\bbr^ny_i$ for all $1\leq i\leq k$ and $\mathcal{D}\cap H\ltimes\bbr^ny_i$ and $\hat f-\hat f g_r.$ Then there exist $x_1\cdots,x_\ell$ such that the conclusion of the theorem holds for $\hat f-\hat f g_r$ and $\vare/3.$ Combining this together with~(\ref{noscapeofmass}) and~(\ref{ubddcontrol2}) and the choice of $r$ we get the theorem.
\end{proof}

\textit{Proof of theorem~\ref{signaturenot(2,2)} and~\ref{uniformnot(2,2)}.} Theorem~\ref{signaturenot(2,2)} is special case of theorem~\ref{uniformnot(2,2)}. The proof of theorem~\ref{uniformnot(2,2)} now goes along the same lines as in~\cite[Section 3.4, 3.5]{EMM1} replacing theorem 3.5 in~\cite{EMM1} by theorem~\ref{unbddequinot(2,2)} above. The proof is based on integrating equation~(\ref{thetatrans}).


\section{The case of signature $(2,2)$.}\label{sec:(2,2)} We now turn to the more interesting case of signature $(2,2).$ The proof is based on the same philosophy however since theorem~\ref{emm1alpha} does not hold for $\alpha_2$ in general the proof is more involved. In fact theorem~\ref{unbddequinot(2,2)} may fail in case of $(2,2)$ in general. It holds however under the Diophantine condition assumed above if we replace $\hat f$ by a modified function $\tilde f.$ After this slight modification the main difficulty is to control the contribution coming from $\alpha_2$ to integrals similar to those considered in theorem~\ref{unbddequinot(2,2)}. In this section we make this reduction and the next section is devoted to the careful study of this contribution. The fact that we consider a slightly different function $\tilde{f}$ is due to existence of exceptional subspaces which have of order of $T^2$ solutions and is a minor point. They may exists even under Diophantine assumption, indeed we will show that if $Q_\xi$ is irrational then there are at most four exceptional subspaces.

Let $Q_\xi$ be an inhomogeneous quadratic form of signature $(2,2)$ with discriminant 1. Recall that the affine subspace $L-v_\xi$ is called exceptional if $L$ is a rational null subspace and $\xi\in L+v_\xi$ for some $v_\xi\in\bbz^4.$ Indeed in this case $L\subset\bbz^4+\xi$ and we will refer to $L$ as an exceptional subspace of $\bbz^4+\xi.$ Let now  $\mathfrak{q}\in\SL_4(\bbr)$ be such that $Q(v)=B(\mathfrak{q}v)$ for all $v\in\bbr^4.$ Let $\Lambda=\mathfrak{q}\bbz^4$ and let $\Lqxi=\mathfrak{q}(\bbz^4+\xi).$ Let $X(\Lqxi)$ be the set of vectors in $\Lqxi$ not contained in $\mathfrak{q}L$ where $L\subset\bbz^4$ for $Q$ defined as above. Define
\begin{equation}\label{mthetatrans}\tilde f(g:\Lqxi)=\sum_{v\in X(\Lqxi)}f(gv)\end{equation}

The following is analogue of~\cite[Theorem 2.3]{EMM2} and will provide us with the upper bound required for the proof of theorem~\ref{(2,2)}. The proof of this theorem is the main technical part of this paper and will occupy the rest of this paper.

\begin{thm}\label{unboundedequi(2,2)}
Let $G, H, K$ and $\{a_t\}$ be as in the section~\ref{sec;inhomo} for the signature $(2,2)$ case. Let $Q_{\xi}$ be a quadratic form of signature $(2,2)$ which is Diophantine. Let $\mathfrak{q}\in\SL_4(\bbr)$ and $\Lqxi$ be as above. Let $\nu$ be a continuous function on $K.$ Then we have
\begin{equation}\label{e:unboundedequi(2,2)}\limsup_{t\rightarrow\infty}\int_K\tilde{f}(a_tk:\Lqxi)\nu(k)dk\leq\int_{G/\Gamma}\hat{f}(g)d\mu(g)\int_K\nu(k)dk\end{equation}
where $\mu$ is the $G$-invariant probability measure on $G/\Gamma$ if the homogenous part, $Q,$ is irrational and the $H\ltimes\bbr^4$-invariant probability measure on the closed orbit $H\ltimes\bbr^4\cdot\Lqxi$ if $Q$ is a rational form.
\end{thm}
\textit{Proof of theorem~\ref{(2,2)}.} Suppose $Q_\xi$ is as in the statement of theorem~\ref{(2,2)}. An argument like that of~\cite[Section 3.4, 3.5]{EMM1} combined with theorem~\ref{unboundedequi(2,2)} gives: If $0\notin(a,b)$ then
\begin{equation}\label{upperbound}\limsup_{T\rightarrow\infty}{N}_{Q,\xi,\Omega}(a,b,T)=\limsup_{T\rightarrow\infty}\tilde{N}_{Q,\xi,\Omega}(a,b,T)\leq\lambda_{Q,\Omega}(b-a)T^2\end{equation}
This upper bound combined with the lower bound obtained by theorem~\ref{dmlowerbound} proves theorem~\ref{(2,2)}.

The proof of theorem~\ref{unboundedequi(2,2)} will extensively utilize results and ideas from in~\cite{EMM1} and~\cite{EMM2}. We will try to use terminologies and notations used in loc. cit for the convenience of the reader. We recall these theorems and terminologies when we need them. Let us start with the following

\begin{thm}\label{alpha13}
Let $\{a_t\}$ and $K$ be as in theorem~\ref{unboundedequi(2,2)}. Let $\Lambda$ be any lattice in $\bbr^4$ then for $i=1,3$ and any $\vare>0$
\begin{equation}~\label{e:alpha13}\sup_{t>0}\int_K\alpha_i(a_tk\Lambda)^{2-\vare}dk<\infty\end{equation}
Hence there exists a constant $c$ depending on $\vare$ and $\Lambda$ such that for all $t>0$ and $0<\delta<1$
\begin{equation}\label{measureestimate13}|\{k\in K\h:\h \alpha_i(a_tk\Lambda)>\frac{1}{\delta}\}|<c\delta^{2-\vare}\end{equation}
\end{thm}

\begin{proof}
The first assertion is proved in~\cite[section 5]{EMM1}. The second assertion is a consequence of the first assertion and Chebechev's inequality.
\end{proof}

Such statement for $\alpha_2$ however, does not hold in general. Hence in order to control the integral on the left hand side of~(\ref{e:unboundedequi(2,2)}) we need to study the contribution coming from $\alpha_2$ to this integral under the imposed Diophantine condition. Let us recall the following standard

\begin{lem}\label{linearalgebra}{\rm{(\cite[Lemmas 2.1, 10.2]{EMM2})}}
Let $Q$ be a homogeneous quadratic form of signature $(2,2)$. Then $\bigwedge^2\bbr^4$ decomposes into two ${\rm{SO}}(Q)$-invariant subspaces $V_1$ and $V_2.$ Let $\pi_i$ denote the projection $\bigwedge^2\bbr^4\rightarrow V_i.$

\begin{itemize}
\item[(i)] The spaces $V_1$ and $V_2$ are orthogonal with respect to the bilinear form $Q^{(6)}(v,w)=v\wedge w$ on $\bigwedge^2\bbr^4.$ The restriction of $Q^{(6)}$ to $V_i$ has signature $(2,1).$
\item[(ii)] The pair $(V_1,V_2)$ determines $Q$ up to proportionality and the map $f$ which takes $(V_1,V_2)$ to $Q/$proportionality is a rational map defined over $\bbq.$
\item[(iii)] If $V_1$ is rational and the restriction of $Q^{(6)}$ to $V_1$ splits over $\bbq$ then $f(V_1,V_1^{\perp})$ is a split form over $\bbq.$
\item[(iv)] Let $L$ be a two dimensional subspace of $\bbr^4$ and let $v_1,v_2$ be a basis for $L.$ Then the restriction of $Q$ to $L$ is identically zero if and only if $\pi_1(v_1\wedge v_2)=0$ or $\pi_2(v_1\wedge v_2)=0.$
\end{itemize}

\end{lem}

Let $\Lambda$ be a lattice in $\bbr^4.$ If $L$ is a $2$-dimensional $\Lambda$-rational subspace of $\bbr^4$ we let $v_1,v_2$ be an integral basis for $L\cap\Lambda.$ We let $v^{L}=v_1\wedge v_2.$ We will refer to $\bbz^4$-rational subspaces as rational subspaces.

\begin{definition}\label{quasinull}
Let $Q$ be a quadratic form of signature $(2,2).$ Fix $0<\mu_1<1.$ A $2$-dimensional rational subspace $L$ of $\bbr^4$ is called $\mu_1$-{\it quasinull with respect to} $Q$ if
\begin{equation}\label{eq:quasinull}\|\pi_1(v^L)\|\|\pi_2(v^L)\|<\mu_1\end{equation}
where $\pi_i$'s are the projections corresponding to $Q$ defined in lemma~\ref{linearalgebra}.
\end{definition}
If $L$ is a quasinull subspace and $T/2\leq\|v^L\|\leq T$ then either $\|\pi_2(v^L)\|<C/T$ or $\|\pi_1(v^L)\|<C/T,$ we call $L$ quasinull subspace of the first respectively the second type. In particular null subspaces are quasinull and similar terminology will be used for null subspaces. This indeed depends on the ordered pair $(V_1,V_2).$ We assume this ordering is fixed once and for all.
In the particular case of $Q=B$ we fix the ordering so that the space spanned by $\{x_{11},x_{12}\}$ is of first kind.

The following is a technically involved theorem which is proved in~\cite{EMM2}. It controls the contribution to $\alpha_2$ coming from non-quasinull subspaces.

\begin{thm}\label{notquasinull}{\rm{(\cite[Theorem 2.6]{EMM2})}}
Let $\mu_1>0$ be fixed and let $\mathcal{L}_t(\delta)$ be the set of all non-$\mu_1$-quasinull subspaces $L$ such that for some $k\in K,$ $d(a_tkL)<\delta$. Then there exists $\delta_0=\delta_0(\mu_1,Q)$ such that for all $0<\delta<\delta_0$ and all $t>0$
\begin{equation}\label{e:notquasinull}|\bigcup_{L\in\mathcal{L}_t(\delta)}\{k\in K\h:\h d(a_tkL)<\delta\}|<\delta^{1.04}\end{equation}
\end{thm}

Theorems~\ref{alpha13} and~\ref{notquasinull} reduce the proof of theorem~\ref{unboundedequi(2,2)} to the study of quasinull subspaces. We finish this section by making this reduction explicit and also fixing some notations to be used in the next section.

We first remark that in the proof of theorem~\ref{unbddequinot(2,2)} we need to chose $\mu_1$ carefully, see section~\ref{sec;main} for details. However until that section we will let $0<\mu_1<1$ be any small constant which is fixed through out and hence we will drop that from the notation. Let $\Lqxi$ be as in the statement of theorem~\ref{unboundedequi(2,2)}. For $\Lambda$ any lattice in $\bbr^4$ we let $\alpha_{13}(\Lambda)=\max\{\alpha_1(\Lambda),\alpha_3(\Lambda)\}.$ There are constants $c$ and $r$ depending on $f$ and $\xi$ only, such that for all $0<\delta\ll1$ we have
\begin{equation}\label{thinsets}\{k\in K\h:\h \tilde{f}(a_tk\Lqxi)>\frac{c}{\delta}\}\subset\{k\in K\h:\h \alpha_{13}(a_tk\Lqxi)>\frac{1}{\delta}\}\hh\cup \bfrak_t(\delta)\cup\afrak_t({\delta})\end{equation}
where $\bfrak_t(\delta)=\bigcup_{L\in\mathcal{L}_t(\delta)}\{k\in K\h:\h d(a_tkL)<\delta\}.$ The set $\afrak_t({\delta})$ corresponds to quasinull subspaces and will be described below. Recall first that $Q_\xi$ is an irrational form hence it has at most four exceptional subspaces. That is there are at most four null subspaces $L_i$ for $1\leq i\leq4$ such that if $\xi\in L+v$ for some $v\in\bbq^4$ then $L=L_i$ for some $1\leq i\leq4.$ Let $\mathcal{Q}=\mathcal{Q}(Q_\xi)$ be the set of quasinull subspaces $L$ such that $L\neq L_i$ for $1\leq i\leq4.$ If $B(r)$ denotes the ball of radius $r$ in $\bbr^4$ let
\begin{equation}\label{thinsetnull}\afrak_t({\delta})=\bigcup_{L\in\mathcal{Q}}\{k\in K\h:\h d(a_tkL)<\delta\h\&\h\exists\hh v\in\bbz^4\h\hh\mbox{s. t.}\h a_tk(L+v+\xi)\cap B(r)\neq\emptyset\}\end{equation}
Indeed an estimate like~(\ref{thinsets}), actually for $\hat f$ would hold, by the virtue of~\cite[Lemma 2]{Sc} if $\afrak$ was defined by taking union over all quasinull subspaces. But we have replaced $\hat f$ by $\tilde f$ and this implies we can take the union over $\mathcal{Q}$ instead. To see this consider one of these subspaces e.g. $L_1.$ Assume $\xi\in L_1+w_{1\xi}$ where $w_{1\xi}\in\bbz^4.$ Now if there is $k\in K$ such that~(\ref{thinsetnull}) is satisfied with $L=L_1$ and $v\in\bbz^4.$ Then there is a constant $c_{\xi}\geq1$ depending on $\xi$ such hat $d(a_tkH)<c_\xi\delta$ where $H=\mbox{span}\langle L_1,v+w_{1\xi}\rangle.$ Now either $L_1=L_1+v+\xi$ or $H$ is a 3-dimensional subspace. We have excluded the points in $L_1$ in the definition of $\tilde{f}$ and if the later occurs then $k\in\{k\in K\h:\h \alpha_3(a_tk\Lqxi)>1/c_\xi\delta\}.$ Hence there is $c$ such that (\ref{thinsets}) holds.

Now the proof of theorem~\ref{unboundedequi(2,2)} will be completed if we can show that there is some $\eta>0$ depending on $Q_\xi$ such that $|\afrak_t^L(\delta)|<\delta^\eta,$ for all small $\delta.$ Such bounds in general do not hold. In the next section we will prove such bound for ``most'' quasinull subspaces $L$ under the Diophantine condition.


\section{Contribution from quasinull subspaces}\label{sec:quasinull}
The reductions made in the previous section lead us to the study of quasinull subspaces. We need some further investigations before we can take advantage of the Diophantine condition.

Until further notice we work with the standard form and the lattice $\Lambda=\qfrak\bbz^4.$ We will use the projections $\pi_i$ introduced in lemma~\ref{linearalgebra}. While working with $\Lambda$ these are projections from $\bigwedge^2\bbr^4$ onto $V_i$ where $V_i$'s are $\mbox{SO}(B)$ invariant. Similarly by a quasinull subspace we mean a $\Lambda$ rational $2$-dimensional subspace for which~(\ref{quasinull}) holds with respect to these projections and some $\mu_1<1$.

We start by recalling some notations. Let $\mathcal{Q}=\mathcal{Q}_\xi$ denote the set of quasinull subspaces which are not exceptional subspaces. For $L\in\mathcal{Q}$ we let $\afrak_t^L({\delta})$ be the corresponding set defined in~(\ref{thinsetnull}). Let us recall from~\cite{EMM2} that the subset in the compact group $K$ where the subspace $L$ becomes ``thin" can be approximated by union of at most four rectangles. We recall the precise statement which is tailored for our purpose in here.

\begin{lem}\label{intervals}
Let $L$ be a $\mu$-quasinull subspace of first type.  Let $0<\delta<1$ and $t>0.$ Then there exists $0<c<1$ such that
\begin{equation}\label{e:intervals}R_t^{L,\pm,\pm}(c\delta,c\delta)\subset\{k\in K\h:\h d(a_tkL)<\delta\}\subset \bigcup_{\pm}R_t^{L,\pm,\pm}(c^{-1}\delta,c^{-1}\delta)\end{equation}
where $R_t^L=\itlone\times\itltwo$ are boxes with the following properties
\begin{itemize}
\item[(i)] the intervals $\itli$ have the same center for all $t.$
\item[(ii)] the length $|\itlone|\approx e^{-t}m_t(L)^{-1/2}\delta^{1/2},$ where $m_t(L)=\min_{k\in K}d(a_tkL).$
\item[(iii)] For $0<\eta\ll 1$ either $|\itltwo|\leq\delta^\eta$ or $\|\pi_2(v^L)\|\leq\frac{\delta^{1-2\eta}}{e^t}.$
\end{itemize}
Furthermore if $L$ is a null subspace (of first type) then
\begin{equation}\label{e:intervalsnull}\{k\in K\h:\h d(a_tkL)<\delta\}\subset \itlone\times{\rm{SO}}(2)\end{equation}
where $\itlone$ is an interval with above properties.
\end{lem}

\begin{proof}
Parts (i) and (ii) follow from~\cite[Lemma A.6]{EMM2}. Note that if $L$ is a null subspace as a consequence of lemma \ref{linearalgebra} above we have $\pi_2(v^L)=0$ hence $L$ is invariant by the second factor which implies~(\ref{e:intervalsnull}). To see part (iii) note that the action of $\rm{SO}(2)$ on $V_2$ is via the adjoint representation and $b_t$ expands $e_2\wedge e_4$ by a factor of $e^t$ hence (iii) follows.
\end{proof}

The above lemma reduces our study to the investigation of the sets $R_t^{L}(c^{-1}\delta)=R_t^{L,\pm,\pm}(c^{-1}\delta,c^{-1}\delta).$ In the rest of this section we will always assume that $L$ is of first type.




Recall that the set $\afrak_t^L(\delta),$ which is defined in~(\ref{thinsetnull}), is a subset of $K$ where the subspace $L$ has short vectors. Using theorem~\ref{alpha13} we may reduce to a subset where the shortest vectors are of approximately the same size, this reduction is done as follows; Let $k\in\afrak_t^L(\delta)$ and let $u_1,u_2$ be two primitive vectors in $L$ such that $a_tku_1,\hh a_tku_2$ are successive minima of $a_tkL.$ Let $\vare>0$ be small and let
\begin{equation}\label{badthinsetnull}\afrak_t^{L,1}(\delta,\vare)=\left\{k\in K\h:\h\begin{array}{c}\delta^{1+\vare}<d(a_tkL)<\delta\h\&\h\delta^{\frac{1+\vare}{2}}\leq\|a_tku_i(k)\|<\delta^{\frac{1-\vare}{2}}\h i=1,2  \\ \exists\h v\in\bbz^4\h\mbox{such that}\hh a_tk(L+v+\xi)\cap B(r)\neq\emptyset\end{array}\right\}\end{equation}
Let $\afrak_t^{L,2}(\delta,\vare)=\cup_{L\in\mathcal{Q}}(\afrak_t^L(\delta)\setminus\afrak_t^{L,1}(\delta,\vare)),$ using theorem~\ref{alpha13} we see that there exits $\beta>0$ depending on $\vare$ such that  $|\afrak_t^{L,2}(\delta,\vare)|<\delta^{1+\beta}.$ Hence we need to study the sets $\atld=\afrak_t^{L,1}(\delta,\vare).$ Let $\mathcal{Q}_t(\delta,\vare)$ be the set of quasinull subspaces which are not exceptional subspaces and for which $\atld\neq\emptyset.$

Recall that $L$ is a quasinull subspace of the first type, this means $L$ is ``almost" in $V_1,$ see definition~\ref{quasinull}. Then $\itltwo$ is essentially the entire $\mbox{SO}(2)$ in the second factor of $K.$ Indeed if $L$ is a null subspace then $L$ is in $V_1$ and $\itltwo=\mbox{SO}(2),$ as we remarked in lemma~\ref{intervals}. Our goal in this section is to show that, under the Diophantine condition, we can get a better bound for the measure of those  $k\in\mbox{SO}(2)$ which effect $\tilde{f}$ i.e. there is a constant $\eta$ depending on the Diophantine condition and an appropriate choice of $\vare$ such that for all small enough $\delta$ we have $|\atld|<|\itlone|\hh\delta^\eta$ holds for ``most'' quasinull subspaces of first type $L$, see corollary~\ref{powersaving} for a precise statement.

For simplicity let $\zeta=\qfrak\xi,$ where $\qfrak\in\mbox{SL}_4(\bbr)$ was chosen such that $Q(v)=B(\qfrak v)$ for all $v\in\bbr^4.$ Fix $t>0$ and $0<\delta<1.$ Let $L\in\qtv$ with $T/2\leq\|v^L\|\leq T.$ We also fix a small $0<\vare\ll1$ to be determined later in the course of our analysis. Let $k_\theta\in p_1(\atld)$ and let
\begin{equation}\label{conditionalsets}E_t^L(\delta,k_\theta,\vare)=\left\{k_\phi\in \mbox{SO}(2)\h:\h \begin{array}{c}\delta^{1+\vare}<d(a_t(k_\theta,k_\phi)L)<\delta \\ {\displaystyle \min_{0\neq w\in L\cap\Lambda}}\|a_t(k_\theta,k_\phi)w\|>\delta^{\frac{1+\vare}{2}}\end{array}\right\}\end{equation}
Note that $E_t^L(\delta,k_\theta,\vare)$ is an open set in $\mbox{SO}(2).$ Further we have the successive minima of $a_t(k_\theta,k_\phi)L$ are bounded by $\delta^{\frac{1+\vare}{2}}$ and $\delta^{\frac{1-\vare}{2}}$.

\begin{remark}\label{reductiontheory}Using reduction theory of orthogonal group we see that the lattice $a_t(k_\theta,k_\phi)\Lambda$ is ``narrow'' along $L$ only. To be more precise we have if there is some $\lambda\in\Lambda\setminus L$ such that $a_t(k_\theta,k_\phi)(L+\lambda)\cap B({1}/{\delta^{\frac{1-\vare}{2}}})\neq\emptyset$ then $\lambda\in L.$
\end{remark}

Roughly speaking  the general strategy now is to show that the following  dichotomy holds: either there is one translate of $L$ which always stays ``close'' to the origin, or different translates approach the origin in certain time intervals and after spending some time close to the origin move far away, see proposition~\ref{nondivergence} for the precise statement. If the second possibility holds for a subspace $L$ then we use the special geometry of the lattice $a_t(k_\theta,k_\phi)\Lambda$ to guarantee that in the intermediate times there is no translate of $L$ which intersects a fixed bounded neighborhood of the origin. A quantitative form of this argument is provided in proposition~\ref{nondivergence} below. Arguments of this kind are by no means new G.~A.~Margulis used a qualitative version of this argument in his proof of nodivergence for unipotent flows and after him these have been used to provide quantitative versions of nondivergence for polynomial-like maps by several people. We will then use the Diophantine condition to show that the first possibility cannot hold for ``many'' subspaces.

Let us fix some more notations before proceeding. If $k_\phi\in E_t^L(\delta,k_\theta,\vare)$ is given then let $L_\phi=a_t(k_\theta,k_\phi)L,$ more generally for any $x\in\bbr^4$ denote $(L+x)_\phi=a_t(k_\theta,k_\phi)(L+x).$ For $x\in\bbr^4$ let $x_{{L_\phi}^\perp}$ denote the projection of $x$ onto the orthogonal complement of $L_\phi,$ the orthogonal complement is taken with respect to the usual inner product on $\bbr^4.$



\begin{prop}\label{nondivergence}
Let $r>0$ be a constant and let $B(r)$ be the ball of radius $r$ in $\bbr^4$ also keep all the notations from before. Then there exists an absolute constant $c>0$ such that for all $0<\delta\ll1$ one of the following holds

\begin{itemize}
\item[(i)]there exists $\lambda_L\in \Lambda$ such that if for some $\lambda\in\Lambda$ there is $k_\phi\in E_t^L(\delta,k_\theta,\vare)$ for which the plane $(L+\lambda+\zeta)_\phi$ intersects $B(r)$ then $L+\lambda=L+\lambda_L.$

\item[(ii)]for all $\lambda\in\Lambda$ we have $\max_{k_\phi}\|(a_t(k_\theta,k_\phi)[\lambda+\zeta])_{{{L_\phi}^\perp}}\|>\frac{c}{\delta^{{(1-\vare)}/{2}}}.$
\end{itemize}
Furthermore if {\rm{(ii)}} holds then there exists a computable constant $\eta_1>0$ such that for $0<\vare<1/12$ we have
\begin{equation}\label{e:nondivergence}|E_t^L(\delta,k_\theta,\vare,\xi)=\{k_\phi\in E_t^L(\delta,k_\theta,\vare)\hh:\hh\exists \lambda\h{\rm{s.\hh t.}}\h (L+\lambda+\zeta)_\phi\cap B(r)\neq\emptyset\}|<C\delta^{\eta_1}\end{equation}
where $C>0$ is a computable constant.
\end{prop}

\begin{proof}
Assume (ii) fails that is, there exists $\lambda_0$ such that $\|(a_t(k_\theta,k_\phi)[\lambda_0+\zeta])_{{{L_\phi}^\perp}}\|\leq\frac{c}{\delta^{{(1-\vare)}/{2}}}$ for all $k_\phi$ and all $c>0.$ This, for suitable choice of $c,$ implies that $(L+\lambda_0+\zeta)_\phi$ intersects $B(\frac{1}{4\delta^{{(1-\vare)}/{2}}})$ for all $k_\phi.$ Now suppose $(L+\lambda+\zeta)_\phi\cap B(r)\neq\emptyset$ for some $k_\phi$ and $\lambda.$ Now remark~\ref{reductiontheory} guarantees that for $\delta\ll r$ we have $(L+\lambda-\lambda_0)_\phi=L_\phi$ and hence (i) holds. This establishes the first claim.

Assume now that (ii) holds for some $L\in\qtv$ and $k_\theta\in p_1(\atld).$ Denote by $h_\lambda(\phi)=\|(a_t(k_\theta,k_\phi)(v^L\wedge(\lambda+\zeta))\|_m,$ where $\|\h\|_m$ denotes the maximum norm on $\bigwedge^3\bbr^4.$ Note that each function in the definition of $h_\lambda$ is a linear combination of $\{\sin^i\phi\cos^j\phi\h:\h0\leq i,j\leq 3\}.$ Now using~\cite[Proposition 3.4]{KM} we have; There are constant $C_1,\beta_1>0$ such that the function $h_\lambda$ is $(C_1,\beta_1)$ good. It follows from the definition of $(C,\alpha)$-good functions, see~\cite{KM}, that $f_\lambda(\phi)=\|(a_t(k_\theta,k_\phi)(v^L\wedge(\lambda+\zeta))\|$ is $(C_2,\beta_2)$-good for some $C_2,\beta_2>0.$

Recall now that
\begin{equation}\label{3estimates}\begin{array}{c}\max_{k_\phi\in E}\|(a_t(k_\theta,k_\phi)[\lambda+\zeta])_{{{L_\phi}^\perp}}\|>\frac{c}{\delta^{{(1-\vare)}/{2}}}\vspace{.75mm}\\ \delta^{1+\vare}<a_t(k_\theta,k_\phi)<\delta\h\mbox{for all}\h k_\phi\in E_t^L(\delta,k_\theta,\vare,\xi)\vspace{.75mm}\\ f_\lambda(\phi)=\|(a_t(k_\theta,k_\phi)(v^L\wedge(\lambda+\zeta))\|=\|(a_t(k_\theta,k_\phi)v^L\|\|(a_t(k_\theta,k_\phi)[\lambda+\zeta])_{{L_\phi}^\perp}\|\end{array}\end{equation}

Hence we have $\max_\phi f_\lambda(\phi)\geq\frac{c\delta^{1+\vare}}{\delta^{{(1-\vare)}/{2}}}=c\hh\delta^{\frac{1+3\vare}{2}}$ for all $\lambda\in\Lambda.$ For $\lambda\in\Lambda$ define
\begin{equation}\label{marking1}J_\lambda=\{k_\phi\in\mbox{SO}(2)\h:\h f_\lambda(\phi)<\frac{c}{3}\delta^{\frac{1+3\vare}{2}}\}\end{equation}
The sets $J_\lambda$'s are open and an argument like that of the beginning of the proof using the estimates in~(\ref{3estimates}) shows that $J_\lambda$'s are disjoint. Define
\begin{equation}\label{elambda}E_\lambda=\{k_\phi\in E_t^L(\delta,k_\theta,\vare)\hh:\h\h (L+\lambda+\zeta)_\phi\cap B(r)\neq\emptyset\}\end{equation}
Note that for $\phi\in E_\lambda$ we have $f_\lambda(\phi)<r\delta$ hence for $\delta\ll r$ we have $E_\lambda\subset J_\lambda.$ For each $\phi\in J_\lambda$ now let $I_\phi(\lambda)$ be the largest interval containing $\phi$ and contained in $J_\lambda.$ Since $f_\lambda$'s are $(C_2,\beta_2)$-good we have
\begin{equation}\label{goodfunctions}|\{\gamma\in I_\phi(\lambda)\h:\h f_\lambda(\gamma)<r\delta\}|<C(r,c_1,c_2)\delta^{\beta_2\frac{1-3\vare}{2}}|I_\phi(\lambda)|\end{equation}
Let $\vare<1/12$ and we let $\eta_1=\beta_2/4.$ Since $I_\phi(\lambda)$'s cover $J_\lambda$ and are disjoint we have
\begin{equation}\label{markingestimate}|E_\lambda|\leq|\{\phi\in J_\lambda\h:\h f_\lambda(\phi)<r\delta\}|<C\delta^{\eta_1}|J_\lambda|\end{equation}
Note that $E_t^L(\delta,k_\theta,\vare,\xi)\subset\cup_\lambda E_\lambda$ and $E_\lambda\subset J_\lambda$'s are disjoint hence we have
\begin{equation}\label{nondivergencecontrol}|E_t^L(\delta,k_\theta,\vare,\xi)|\leq\sum_\lambda|E_\lambda|\leq C\delta^{\eta_1}\sum_\lambda|J_\lambda|\leq C\delta^{\eta_1}\end{equation}
This finishes the proof of the second claim.
\end{proof}

We now study more extensively the occurrence of case (i) in proposition~\ref{nondivergence}. Thus let $L\in\qtv$ be a quasinull subspace with $T/2\leq\|v^L\|\leq T$ and let $k_\theta\in p_1(\atld)$ and suppose $\lambda_0\in\Lambda$ is such that (i) in proposition~\ref{nondivergence} holds with this $\lambda_0$. Let $0<\eta_2\ll 1$ be small number to be determined later. Define
\begin{equation}\label{e:nondivergence(i)}E_t^L(\lambda_0)=\{k_\phi\in E_t^L(\delta,k_\theta,\vare)\hh:\hh\h (L+\lambda_0+\zeta)_\phi\cap B(r)\neq\emptyset\}\end{equation}
There are two cases to be considered either $|E_t^L(\lambda_0)|\leq\delta^{\eta_2},$ or $|E_t^L(\lambda_0)|>\delta^{\eta_2}.$ The former implies the desired bound on the measure of $\atld$ hence the later is the case which requires a more careful study. We start with the following

\begin{lem}\label{smalltrans}
Let $0<\eta_2<\eta_1$ and the notations and assumptions be as in the above paragraph. In particular assume that (i) in proposition~\ref{nondivergence} holds for some $\lambda_0\in\Lambda$ and further assume that $|E_t^L(\lambda_0)|>\delta^{\eta_2}.$ Then $\|(\lambda_0+\zeta)_{{L}^\perp}\|<\frac{2r\delta^{1-\eta_2}}{T},$ where $r$ is as in the statement of of proposition~\ref{nondivergence}.
\end{lem}

\begin{proof}
Note that the representation of $\mbox{SO}(2,2)$ on $\bigwedge^3\bbr^4$ is the contragredient representation to the standard representation of $\mbox{SO}(2,2)$ on $\bbr^4$. More precisely,  let $e_{ijk}=e_i\wedge e_j\wedge e_k$ for $1\leq i,j,k\leq 4.$ Then action of $K_1$ fixes planes spanned by  $\{e_{134},e_{123}\}$ and $\{e_{341},e_{342}\}$ and similarly $K_2$ fixes the planes spanned by $\{e_{123},e_{124}\}$ and $\{e_{341},e_{342}\}.$ \vspace{1mm}Further $a_te_{123}=e^{-t}e_{123},$ $a_te_{124}=e_{124},$ $a_te_{341}=e_{341}$ and $a_te_{342}=e^te_{342}.$ Now let $w=(k_\theta,e)(v^L\wedge(\lambda_0+\zeta)),$ $a=\|(\lambda_0+\zeta)_{L^\perp}\|$ and for any $k_\phi\in E_t^L(\lambda_0)$ let $a_\phi=\|(a_t(k_\theta,k_\phi)[\lambda_0+\zeta])_{{L_\phi}^\perp}\|.$ Recall that $T/2\leq\|v^L\|\leq T$ hence $aT/2\leq\|w\|\leq aT.$ As was mentioned above $K_2$ acts on the plane spanned by $\{e_{123},e_{124}\}$ by rotation on $\bbr^2$ and $e_{123}$ is the contracting direction of $\{a_t\}.$ Also note that we are assuming that $|E_t^L(\lambda_0)|>\delta^{\eta_2}.$ Hence there exist $k_\phi\in E_t^L(\lambda_0)$ such that
\begin{equation}\label{e:smalltrans}\frac{aT\delta^{\eta_2}}{2}\leq\|a_t(k_\theta,k_\phi)[v^L\wedge(\lambda_0+\zeta)]\|=a_\phi\|a_t(k_\theta,k_\phi)v^L\|\end{equation}
Note now that we have $a_\phi\leq r$ and $\|a_t(k_\theta,k_\phi)v^L\|\leq\delta.$ So we get $a\leq\frac{2r\delta^{1-\eta_2}}{T}$ as we wanted to show.
\end{proof}

Before proceeding let us draw the following corollary. This will be used in the proof of theorem~\ref{unboundedequi(2,2)} to control the contribution of ``small'' subspaces.

\begin{cor}\label{smallsubspaceset}
Let $\eta_1$ be as in proposition~\ref{nondivergence} and let $\eta_2<\eta_1<{1}/{4}.$ Then for any $M>1$ there is a $\delta_0=\delta_0(M,\Lambda)$ such that if $\delta<\delta_0$ then for any $L$ with $\|v^L\|<M$ either $E_t^L(\delta,k_\theta,\vare,\xi)<\delta^{\eta_2}$ or there is some $\lambda\in\Lambda$ such that $\lambda+\zeta\in L.$
\end{cor}

\begin{proof}
Suppose for some $L$ we have $E_t^L(\delta,k_\theta,\vare,\xi)>\delta^{\eta_2}.$ Hence the conditions in lemma~\ref{smalltrans} are satisfied for $L$ and some $\lambda\in\Lambda.$ This implies $\|(\lambda+\zeta)_{{L}^\perp}\|<\frac{2r\delta^{1-\eta_2}}{T}\leq 2r\delta^{1-\eta_2}.$ Now the assertion follows from discreteness of $\Lambda$ together with the assumption that $\|v^L\|<M.$
\end{proof}

It is now more convenient to work with $Q$ and the standard lattice $\bbz^4.$ Hence from now until the end of this section by a quasinull subspace we mean a $\bbz^4$-rational quasinull subspace.

We are now ready to use the Diophantine condition. Let us start by recalling the following theorem which is proved in~\cite[section 10]{EMM2}. This theorem deals with the Diophantine properties of the homogeneous part. It says ``most" quasinull subspaces are null subspaces of a rational approximation of $Q$ and in particular if $Q$ is not EWAS then there are not ``many'' quasinull subspaces. The precise statement is as follows

\begin{thm}\label{approximation}{\rm{(\cite[Proposition 10.11]{EMM2})}}
There exists an absolute constant $\rho>0$ such that the following holds: Suppose $Q$ is any irrational form of signature $(2,2)$. Then for every sufficiently small $\tau_2>0$ and every $T>2$ one of the following holds:
\begin{itemize}
\item[(i)] The number of quasinull subspaces of $Q$ of norm between $T/2$ and $T$ is $O(T^{1-\tau_2}).$
\item[(ii)] There exists a split integral form $Q'$ with coefficients bounded by a fixed power of $T^{\tau_2}$ and $1\leq\lambda\in\bbr$ satisfying $\|Q-\frac{1}{\lambda}Q'\|\leq T^{-\rho}$ such that the number of quasinull subspaces of $Q$ with norm between $T/2$ and $T$ which are not null subspaces of $Q'$ is $O(T^{1-\tau_2}).$
\end{itemize}
\end{thm}

We refer to section 10 of~\cite{EMM2} and also appendix~\ref{sec:counting} of this paper for a more careful analysis of this theorem. The main ingredient in the proof is the system of inequalities from ~\cite{EMM1}. The main difference is that these inequalities are applied to a certain dilated lattice in $\bigwedge^2\bbr^4=\bbr^6.$

Let us now outline the rest of the proof. Let $Q_\xi$ be the inhomogeneous quadratic form as in statement of theorem~\ref{(2,2)}. We will apply the above theorem with appropriate parameters $\rho,\tau_2$ to be determined later. Let $T\geq2$ be given. Now if (i) above holds, which is always the case if $Q$ is not EWAS, then we already have a good control on the number of quasinull subspaces in question and we will get the desired control on the measure of the set $\cup_L\atld.$ Hence we may assume (ii) above holds. Again if $L$ is not a null subspace of the appropriate approximation of $Q$ we proceed as in case (i). So we need to consider the contribution from quasinull subspaces which are null subspaces of some rational approximation. In this case using lemma~\ref{nondivergence} we will be reduced to the case where only one translate of $L$ has contribution. We then will use the Diophantine property of $\xi$ and get a control on the number of such subspaces, this will complete the proof.

We need to fix some more notations before proceeding with the above outline. If $Q$ is a rational form we may choose $\mu_1$ small enough such that all $\mu_1$-quasinull subspaces are null subspaces. Also in this case, replacing $Q$ by a scalar multiple, we may and will assume that $Q$ is a primitive integral form.

Let $T\geq2$ be a fixed number. Recall from theorem~\ref{approximation} that there are two possibilities for quasinull subspaces. The case which requires more study is case (ii) so let us assume we are in this case. Let $Q_T=Q'$ (resp. $Q_T=Q$) if $Q$ is irrational form (resp. if $Q$ is split integral form) where $Q'$ is given as in (ii) of theorem~\ref{approximation}. Let $\qtd$ be the set of all such quasinull subspaces (resp. null subspaces if $Q$ is rational) which are not exceptional subspaces and such that $\afrak^2_\delta(\bullet,\vare)$ is nonempty for them. Let $L\in\qtd$ be such subspace. We will further assume that $L$ is of first type and that $T/2\leq\|v^L\|\leq T.$.

Since $Q_T$ is a split integral form, after possibly multiplying by a scalar bounded by a fixed power of $T^{\tau_2}$, there exists a non singular integral matrix such that $Q_T(v)=B(\mathfrak{p}v).$ Further theorem~\ref{approximation} guarantees that we may choose $\mathfrak{p}$ such that its entries are bounded by a fixed power of $T^{\tau_2}.$ Recall that the null subspaces of $B$ are of two types, based on the fact that the corresponding vector $v^L$ is in $V_1$ or $V_2.$ From now on by a null space of first kind for $B$ we mean


\noindent
\textit{Null subspaces of first type:} These are subspaces which are orbits of $\left(\begin{array}{cc}x_{11} & x_{12}\\ 0 & 0\end{array}\right)$ under $\SL_2\times\SL_2.$


Trough out the rest of the section, by a null subspace we mean a rational null subspace of the first type. Let now $M$ be a rational null subspace of first type for $B$. Such subspaces are characterized as annihilator of primitive row vectors. Hence an integral basis for $M$ is $\left\{\left(\begin{array}{cc}m & 0\\ n & 0\end{array}\right),\left(\begin{array}{cc}0 & m\\ 0 & n\end{array}\right)\right\},$ where $\mbox{gcd}(m,n)=1,$ we will refer to this basis as standard integral basis for $M.$

Recall that $L\in\qtd$ with $T/2\leq\|v^L\|\leq T.$ Since $L$ is a null subspace of $Q_T$ the subspace $M=\pfrak L$ is a null subspace of $B.$ Now let $\{v_1,v_2\}$ be the standard basis for $M$ and let $w_i$ be the unique primitive integral multiple of $\pfrak^{-1}v_i.$ Then $\{w_1,w_2\}$ is a basis for $L.$ Further since $\pfrak$ is an integral matrix whose entries are bounded by a fixed power of $T^{\tau_2}$ we have $T^{1/2-\tau_3}\leq\|w_i\|\leq T^{1/2+\tau_3},$ where $\tau_3$ is a fixed multiple of $\tau_2.$ We will refer to the basis constructed in this way as $\tau_3$-{\it round basis} for $L.$


Fix $0<\eta_2,\vare,\tau_2\ll1$ small parameters. These will be determined later. Recall that $\qfrak\in\SL_4(\bbr)$ is chosen such that $Q(v)=B(\qfrak v)$.\vspace{.5mm}
Now let $L$ be as above, further assume that there is $k_\theta\in p_1(\mathfrak{A}_t^{\qfrak L}(\delta,\vare))$ such that $|E_t^{\qfrak L}(\delta,k_\theta,\vare,\xi)|>\delta^{\eta_2},$ in particular (i) of proposition~\ref{nondivergence} holds for some $\lambda\in\Lambda$.  Fix $\{w_1,w_2\}$ a $\tau_3$-round. Note that $\tau_3$ is fixed when $\tau_2$ is chosen.

Let $\qfrak^{-1}\lambda=v\in\bbz^4.$ Now, using lemma~\ref{smalltrans}, there is a constant $c=c(\qfrak)$ such that $\|(v+\xi)_{L^\perp}\|<\frac{c(\qfrak)\delta^{1-\eta_2}}{T}.$ This and the fact that $L$ is a null subspace of $Q_T$ give
\begin{equation}\label{smallinnerpro}|\langle w_i\hh,v+\xi\rqt|\leq T^{1/2+\tau_3}\cdot\frac{c\delta^{1-\eta_2}}{T}=\frac{c\delta^{1-\eta_2}}{T^{1/2-\tau_3}}\end{equation}
where $c$ is an absolute constant depending on $Q.$ So $\{\langle w_i\hh,\xi\rqt\}\leq\frac{c\delta^{1-\eta_2}}{T^{1/2-\tau_3}},$ where $\{\h\}$ denotes the distance to the closest integer. Let us now collect the result of the above discussion in the following

\begin{lem}\label{onetransestimate}
Let $\vare$ and $\tau_2$ be small and let $L\in\qtd$. Let $k_\theta\in p_1(\afrak_t^{\qfrak L}(\delta, \vare))$ and suppose that $|E_t^{\qfrak L}(\delta,k_\theta,\vare,\xi)|>\delta^{\eta_2}$ with $\eta_2$ as in lemma~\ref{smalltrans}. In particular (i) in proposition~\ref{nondivergence} holds for $\qfrak L$ and some $\lambda\in\Lambda.$ Then $\{\langle w_i\hh,\xi\rqt\}\leq\frac{c\delta^{1-\eta_2}}{T^{1/2-\tau_3}},$ where $\tau_3$ is a fixed multiple of $\tau_2$ as above. Further since $\{w_1,w_2\}$ is the image of standard basis $\{v_1,v_2\}$ of $M=\pfrak L$ then we have $\{\langle v_i\hh,\pfrak\xi\rb\}\leq\frac{c\delta^{1-\eta_2}}{T^{1/2-\tau_3}}.$
\end{lem}

This lemma brings us to the situation where we can now use the Diophantine property of the vector $\xi.$

As we mentioned in the brief outline following theorem~\ref{approximation}, loc. cit. deals with the Diophantine properties of $Q.$ If $Q$ fails to have desired Diophantine condition $\xi$ needs to be a Diophantine vector thanks to our assumption on $Q_\xi.$ In what follows we will make use of this assumption and control the number of quasinull subspaces for which lemma~\ref{onetransestimate} can hold.

The following is a simple consequence of the definition~\ref{d:farfromrational} i.e. the Diophantine condition. The proof is easy, we include the proof for the sake of completeness also for later references.

\begin{lem}\label{lineardiophantine}
Let $x=(x_1,\cdots,x_n)\in\bbr^n$ be a $\kappa$-Diophantine vector. Then for any $\kappa'>(n+1)\kappa+1$ we have: For all $\beta>0$ and all $0<\gamma\ll1,$ if $A\in{\rm{GL}}_n(\bbq)$ is such that $\max_{ij}\{|A_{ij}|,|(A^{-1})_{ij}|\}<\gamma^{-\beta},$ then $A\hh x$ is $(\kappa'\beta)$-Diophantine.
\end{lem}

\begin{proof}
Note that the inverse matrix $A^{-1}$ is a rational matrix whose entries have denominators bounded by $C_1\gamma^{-n\beta}$ where $C_1$ is an absolute constant depending on the norm on $\bbr^n.$ Let $v=(\frac{p_1}{q_1},\cdots,\frac{p_n}{q_n})$ be any rational vector with $\max_i|q_i|<\gamma^{-\beta}$ such that $\max_i|(Ax-v)_i|<C'\gamma^{\kappa'\beta}.$ The vector $v'=A^{-1}v=(\frac{r_1}{s_1},\cdots,\frac{r_n}{s_n})$ is a rational vector with $\max_i|s_i|<C_2\gamma^{-(n+1)\beta}$ where $C_2$ is an absolute constant. We have
\begin{equation}\label{farfromrational2}\max_i|(x-v')_i|\leq\|A^{-1}\|\|Ax-v\|<C_3\gamma^{(\kappa'-1)\beta}\h\h\mbox{and}\h\h|s_i|<C_2\gamma^{-(n+1)\beta}\end{equation}
Now if one takes $\kappa'>(n+1)\kappa+1$ one gets a contradiction.
\end{proof}

\begin{remark}\label{r;lineardiophantine} Arguing as in the proof of lemma~\ref{lineardiophantine} we can prove the following; For any $\tau_1>0$ we can choose $0<\tau=\tau(\tau_1)$ such that if $\tau_2$ in theorem~\ref{approximation} is less than $\tau$ then for all $T>2$ we have: If $\max_i|(\pfrak\xi)_i-\frac{p_i}{q_i}|< C(\xi)T^{-2\kappa\tau_1}$ then $|q_i|>T^{\tau_1}$ for all $1\leq i\leq4.$
\end{remark}


As before, let $\{w_1,w_2\}$ be a basis for $L$ which is the image of the standard basis $\{v_1,v_2\}$ of $M=\pfrak L.$ Replacing $\tau_3$ by $2\tau_3$ if necessary, we have $T^{1-\tau_3}\leq\|v_i\|^2=m^2+n^2\leq T^{1+\tau_3}.$ The following is a consequence of lemma~\ref{onetransestimate} and remark~\ref{r;lineardiophantine}.

\begin{prop}\label{fewbadsubspaces}
Let $\eta_1$ be as in proposition~\ref{nondivergence} and let $\eta_2<\eta_1.$ Then for any $0<\tau_1\ll1$ we can choose $\tau_2$ small enough such that the number of quasinull subspaces $L$ in $\qtd$ with $T/2\leq\|v^L\|\leq T$ for which there exists some $k_\theta\in p_1(\afrak_t^{\qfrak L}(\delta,\vare))$ such that $|E_t^{\qfrak L}(\delta,k_\theta,\vare,\xi)|>\delta^{\eta_2}$ is $O(T^{1-\tau_1}).$
\end{prop}

\begin{proof}
We may as we will assume that $L$ is of first type. We continue to use the notations as in lemma~\ref{onetransestimate} in particular let $M=\pfrak L$ and let $\{v_1,v_2\}$ be the standard basis for $M.$ Then loc. cit. implies that $\{\langle v_i,\pfrak\xi\rb\}\leq\frac{c\delta^{1-\eta_2}}{T^{1/2-\tau_3}}.$ Recall that $v_1=(m,0,n,0)$ and $v_2=(0,m,0,n)$ where $\mbox{gcd}(m,n)=1$ and we have
\begin{equation}\label{normround}T^{1-\tau_3}\leq \|v^M\|=m^2+n^2\leq T^{1+\tau_3},\end{equation}

where $\tau_3$ is a fixed  multiple of the constant $\tau_2$ appearing in theorem~\ref{approximation}. We also note that $\langle v_1\hh,\pfrak\xi\rb=m(\pfrak\xi)_4-n(\pfrak\xi)_2$ and $\langle v_2\hh,\pfrak\xi\rb(\pfrak\xi)_1-m(\pfrak\xi)_3.$

Let $\tau'_1\ll\frac{1}{32\kappa}$ be chosen. Then choose $\tau_2$ in theorem~\ref{approximation} such that $\tau_2<\tau(\tau'_1),$ where $\tau(\tau')$ is given as in remark~\ref{r;lineardiophantine}. Hence we have; If $\max_i|(\pfrak\xi)_i-\frac{p_i}{q_i}|< C(\xi)T^{-2\kappa\tau'_1}$ then $|q|>T^{\tau'_1}.$ Taking $\tau_2$ even smaller we may and will assume that $\tau_3<\tau'_1/3.$

Now by lemma~\ref{lineardiophantine} we see that there is at most one primitive vector $(a,b)$ with $\|(a,b)\|<T^{\tau'_1}$ such that;
$\mbox{There exists}\h\h(\frac{p_1}{q_1},\frac{p_2}{q_2})\in\bbq^2\hspace{3mm}\mbox{with}\h\h\h|q_i|<T^{\tau'_1}\h\h\mbox{and}$
\begin{equation}\label{farfromrational2}\h\h\begin{array}{c}|a(\pfrak\xi)_1-b(p\frak\xi)_3-\frac{p_1}{q_1}|<C'T^{-8\kappa\tau'_1}\\
|b(\pfrak\xi)_4-a(p\frak\xi)_2-\frac{p_1}{q_1}|<C'T^{-8\kappa\tau'_1}\end{array}\end{equation}
Where $C'=C'(\xi)$ is an absolute constant depending on. If we choose $\delta$ small enough our assumption in the proposition and our choices of $\tau'_1$ and $\tau_2$ imply that
\begin{equation}\label{e:fewbadsubs}\{\langle v_i,\pfrak\xi\rb\}\leq\frac{C'T^{-8\kappa\tau'_1}}{3}\h\h\mbox{for}\h\h i=1,2\end{equation}
Divide the range $T^{1-\tau_3}\leq \|(m,n)\|\leq T^{1+\tau_3}$ into boxes of size $T^{\tau'_1}\times T^{\tau'_1}$. Now~(\ref{farfromrational2}) implies that; From each of these boxes there is at most one primitive integral vector $(m,n)$ which can satisfy~(\ref{e:fewbadsubs}). Hence the number of $(m,n)$ such that~(\ref{e:fewbadsubs}) holds is $O(T^{1+\tau_3-\tau'_1}).$ The conclusion of the proposition thus holds with $\tau_1=\tau'_1/2.$
\end{proof}

Let us now summarize, in the following corollary, what we proved in this section. This is the precise formulation of what we called ``{\it our goal}\h" in the beginning of this section.

\begin{cor}\label{powersaving}
There are positive constants $\eta$ and $\tau$ and also constants $C_1,C_2>0$ which depend on $Q_\xi$ such that, if $0<\delta\ll1$ is small enough then for all $T>2$ the number of quasinull subspaces $L$ in $\qtv$ with $T/2\leq\|v^L\|\leq T$ for which we have $|I_t^{\pi_2(\qfrak L)}(\delta)|>C_1\delta^{\eta}$ is at most $C_2T^{1-\tau}.$
\end{cor}

\begin{proof}
Let $\eta_1$ be as in proposition~\ref{nondivergence} and let $\eta_2<\eta_1.$ Using proposition~\ref{fewbadsubspaces} there exists $\tau_1,\tau_2$ such that the number of $L\in\qtd$ with $T/2\leq\|v^L\|\leq T$ for which we have $|E_t^{\qfrak L}(\delta,k_\theta,\vare,\xi)|>\delta^{\eta_2}$ for some $k_\theta\in p_1(\afrak_t^{\qfrak L}(\delta,\vare)),$ is $O(T^{1-\tau_1}).$ On the other hands using theorem~\ref{approximation} we see that the number of quasinull subspaces with $T/2\leq\|v^L\|\leq T$ and $\atld\neq\emptyset$ which are not in $\qtd$ is $O(T^{1-\tau_2}).$ The corollary follows with $\eta=\eta_1/2$ and $\tau=\min\{\tau_1,\tau_2\}$.
\end{proof}


\section{Proof of theorem~\ref{unboundedequi(2,2)}.}\label{sec;main}
We will complete the proof of theorem~\ref{unboundedequi(2,2)} in this section. Before proceeding to the proof we need the following two statements.

\begin{lem}\label{exceptional}
If $Q_\xi$ is an irrational $(2,2)$ form. Then the number of $2$-dimensional null subspaces, say $L$, of $Q$ such that $\xi\in v+L$ for some $v\in\bbz^4$ is at most four.
\end{lem}

\begin{proof}
First note that using lemma 10.3 in~\cite{EMM2} we may and will assume that $Q$ is a rational form. In this case we will actually show there are at most two such subspaces. In order to see this suppose that there are two null rational subspaces of the same type, say $L_i$ for $i=1,2,$ such that $\xi\in v_i+L_i$ where $v_i\in\bbz^4.$ Now since $L_i$'s are of the same type they are transversal. Let $\{w^i_1,w^i_2\}$ be an integral basis for $L_i,$ then $\langle\xi\hh,w^i_j\rb\in\bbz$ for $i,j=1,2.$ This thanks to the transversality of $L_i$'s implies that $\xi$ is a rational vector which is a contradiction.
\end{proof}

The following is essential to the proof of theorem~\ref{unboundedequi(2,2)} and proved in appendix~\ref{sec:counting}.

\begin{prop}\label{counting}
The number of quasinull subspaces with norm between $T/2$ and $T$ is $O(T).$
\end{prop}

As we mentioned this will be proved in appendix~\ref{sec:counting}. However for the time being let us remark that in the case where $Q$ is a split integral form this is immediate. Indeed in that case we are dealing with null subspaces and hence we need to show this for $Q=B$. As we observed however the null subspaces of $B$ are classified by primitive integral vectors $(m,n).$ Now if $L$ is a null subspace of either type which corresponds to $(m,n)$ then $\|v^L\|=m^2+n^2.$ Thus the result is obvious.

We now turn into the proof of theorem~\ref{unboundedequi(2,2)}.

\textit{Proof of theorem~\ref{unboundedequi(2,2)}:} We may as we will assume that $\tilde{f}$ is non negative. Now define
\begin{equation}\label{cusp}A(r)=\{\Delta\in X\h:\h \alpha_1(\Delta)>r\}\end{equation}
Let $g_r$ be a continuous function on $X$ such that $g_r(\Delta)=0$ if $\Delta\notin A(r),$ $g_r(\Delta)=1$ for all $\Delta\in A(r+1)$ and $0\leq g_r(\Delta)\leq1$ if $r\leq\alpha_1(\Delta)\leq r+1.$ Let the constant $c,$ depending on $f$ and $\xi,$ be as in~(\ref{thinsets}) and let $s\in\bbn$ be a large number. Define
\begin{equation}\label{cutfunction}(\tilde fg_r)_s^\leq=(\tilde fg_r)\chi_{\{\Delta\h:\h \tilde fg_r(\Delta)\leq c2^s\}}\hspace{2mm}\mbox{and}\hspace{2mm}(\tilde fg_r)_s^>=(\tilde fg_r)\chi_{\{\Delta\h:\h \tilde fg_r(\Delta)>c2^s\}}\end{equation}

Let now $\mu$ be as in the statement of the theorem. Since $\tilde f-\tilde fg_r$ is bounded and continuous, theorems~\ref{bequidistribution1} and~\ref{bequidistribution2} imply that
\begin{equation}\label{boundedpart}\limsup_{t\rightarrow\infty}\int_K(\tilde f-\tilde fg_r)(a_tk\Lqxi)\nu(k)dk\leq\int_{G/\Gamma}\hat{f} d\mu\int_K\nu dk\end{equation}
Hence theorem~\ref{unboundedequi(2,2)} will be proved if we show that: Given $\epsilon>0$ we can choose $r_0$ such that if $r>r_0$ then $\limsup_t\int_K\tilde fg_rdk<\epsilon.$

Let now $\epsilon>0$ be an arbitrary small number. Let us first control the contribution of $(\tilde fg_r)_s^>$ when $s$ is large enough. Let $M>0$ be a number fixed for now and large enough such that $\frac{1}{M}<\frac{\epsilon}{6C},$ where $C$ is a universal constant appearing in~(\ref{unboundedintegral4}) and in particular is independent of $\mu_1$ in the definition of quasinull subspaces. Further assume that $M$ is large enough such that all exceptional subspaces have norm less than $M.$ Let $\mu_1$ in the definition of quasinull subspace be small enough such that for all $L\in\qtv$ with $\|v^L\|<M$ if $L$ is quasinull then $L$ is null. Let $s>0$ be large enough such that the conclusion of theorems~\ref{alpha13} and~\ref{notquasinull}, for this $\mu_1$ which we chose, as well as corollaries~\ref{smallsubspaceset} and~\ref{powersaving} hold for $\delta=1/2^j$ whenever $j>s.$ Recall now that we have
\begin{equation}\label{thinsetsrevisted}\{k\in K\h:\h \tilde{f}(a_tk\Lxi)>c2^j\}\subset\{k\in K\h:\h \alpha_{13}(a_tk\Lxi)>2^j\}\hh\cup \bfrak_t(\frac{1}{2^j})\cup\afrak_t(\frac{1}{2^j})\end{equation}
where $\afrak_t(\frac{1}{2^j})$ and $\bfrak_t(\frac{1}{2^j})$ are as in~(\ref{thinsets}). Let $\mathfrak{C}_t(\frac{1}{2^j})=\{k\in K\h:\h \alpha_{13}(a_tk\Lxi)>2^j\}$. We have
\begin{equation}\label{unboundedintegral1}\int_Kh_s^>(a_tK\Lqxi)dk\leq\sum_{s<j<t'}2^j(|\mathfrak{C}_t(\frac{1}{2^j})|+|\mathfrak{B}_t(\frac{1}{2^j})|+|\mathfrak{A}_t(\frac{1}{2^j})|)\end{equation}

For any small $\vare>0$ and any quasinull subspace $L$ let $\afrak_t^{L}(\frac{1}{2^j},\vare)$ be as before. Theorems~\ref{alpha13} and~\ref{notquasinull} imply that if $s$ is large enough then for all $0<\vare\ll1/20$ and all $j>s$ we have
\begin{equation}\label{unboundedintegral2}|\mathfrak{C}_t(\frac{1}{2^j})|+|\bfrak_t(\frac{1}{2^j})|+|\afrak_t(\frac{1}{2^j})\setminus\bigcup_{L\in\mathcal{Q}_t(\frac{1}{2^j},\vare)}\afrak_t^{L}(\frac{1}{2^j},\vare)|<\frac{1}{2^{(1+\frac{\vare}{4})j}}\end{equation}
Let $\eta_1<1/4$ be as in the proposition~\ref{nondivergence} and let $\eta=\eta_1/2$. The conclusions of corollaries~\ref{smallsubspaceset} and~\ref{powersaving} hold with this $\eta$ and the corresponding $\tau.$ Let now $\mathcal{Q}^<_t(\frac{1}{2^j},\vare)$ (resp.  $\mathcal{Q}^{\geq}_t(\frac{1}{2^j},\vare)$) be the set of quasinull subspaces in $\mathcal{Q}_t(\frac{1}{2^j},\vare)$ with norm less than $M$ (resp. greater than or equal to $M.$) We have
\begin{equation}\label{unboundedintegral3}\begin{array}{c}|{\bigcup}_{L\in\mathcal{Q}^<_t(\frac{1}{2^j},\vare)}\afrak_t^{L}(\frac{1}{2^j},\vare)|\leq\sum_i\sum_{\frac{e^t}{2^{i+j+1}}}^{\frac{e^t}{2^{i+j}}}|\itl|\leq C_1\sum_i\frac{1}{2^{(1+\eta)j+\frac{i}{2}}}  \\ |{\bigcup}_{L\in\mathcal{Q}^\geq_t(\frac{1}{2^j},\vare)}\afrak_t^{L}(\frac{1}{2^j},\vare)|\leq \sum_i\sum_{\frac{e^t}{2^{i+j+1}}}^{\frac{e^t}{2^{i+j}}}|\itl|\leq C_2\sum_i(\frac{1}{2^{(1+\eta)j+\frac{i}{2}}}+\frac{e^{-\tau t}}{2^{(\frac{1}{2}-\tau)i+(1-\tau)j}})\end{array}\end{equation}
where $C_1$ and $C_2$ are absolute constant independent of $\mu_1$. The inequality in the first line above follows from corollary~\ref{smallsubspaceset} and the fact that the definition of $\mathcal{Q}_t(\frac{1}{2^j},\vare)$ excludes exceptional subspaces. The inequalities in the second line follow from corollary~\ref{powersaving}. We now have

\begin{equation}\label{unboundedintegral4}\sum_{j>s}2^j|\bigcup_{L\in\mathcal{Q}_t(\frac{1}{2^j},\vare)}\afrak_t^{L}(\frac{1}{2^j},\vare)|\leq C(\frac{1}{2^s\eta}+\frac{1}{M})\end{equation}
here $C'$ is an absolute constant which depends on $C$ in~(\ref{unboundedintegral4}). This inequality together with~(\ref{unboundedintegral2}) gives
\begin{equation}\label{unboundedintegral5}\int_K(\tilde fg_r)_s^>(a_tK\Lqxi)dk\leq \frac{C'}{2^{\frac{\vare s}{4}}}+C(\frac{1}{2^{\eta s}}+\frac{1}{M})\end{equation}

Recall that $M$ was chosen such that $\frac{C}{M}<\epsilon/6.$ Now we choose $s$ large enough such that the right hand side of~(\ref{unboundedintegral5}) is less than $\epsilon/2.$ The above estimate holds for all $r.$

As we mentioned in the proof of theorem~\ref{unbddequinot(2,2)}:  There exists $r_0=r_0(s,\epsilon)$ such that if $r>r_0$ then $\mu(A(r))<\epsilon/2^{s+1}.$ Let $r>r_0(s,\epsilon).$ We have
\begin{equation}\label{boundedincusp}\limsup_{t\rightarrow\infty}\int_K(\tilde fg_r)_s^\leq(a_tk\Lqxi)\nu(k)dk\leq 2^s\limsup_{t\rightarrow\infty}\int_Kg_r(a_tk\Lqxi)\nu(k)dk\leq\epsilon/2\end{equation}

Thus~(\ref{boundedincusp}) and~(\ref{unboundedintegral5}) give: If $r>r_0(s,\vare)$ then
\begin{equation}\label{unboundedintegralfinal}\limsup_{t\rightarrow\infty}\int_K\tilde fg_r(a_tk\Lqxi)dk<\epsilon\end{equation}
This finishes the proof of theorem~\ref{unboundedequi(2,2)}.


\section{Proof of theorem~\ref{(2,1)}}\label{sec;(2,1)}

The proof of theorem~\ref{(2,1)} is very similar to that of theorem~\ref{(2,2)}. Indeed our study in this case is simpler as we are dealing with the case where the homogeneous part $Q$ is rational. Hence we only need to consider the contribution of null subspaces to the counting function.

As before let $\mathfrak{q}\in\SL_3(\bbr)$ be such that $Q(v)=B(\mathfrak{q}v)$ for all $v\in\bbr^3.$ Since $Q$ is rational we may assume $\mathfrak{q}$ is in $\mbox{PGL}_3(\bbq),$ let $\mathfrak{p}\in\mbox{GL}_3(\bbq)$ be a representative for $\mathfrak{q}.$ Let $\Lambda=\mathfrak{q}\bbz^3$ and define $\Lqxi=\mathfrak{q}(\bbz^3+\xi).$ As in section~\ref{sec:(2,2)} let $X(\Lqxi)$ be the set of vectors in $\Lqxi$ not contained in $\mathfrak{q}L$ where $L\subset\bbz^3$ is an exceptional (one dimensional) subspace for $Q$. An argument like that in lemma~\ref{exceptional} shows there are at most $3$ such subspaces when $Q$ is rational and $\xi$ is an irrational vector. For any continuous compactly supported function $f$ on $\bbr^3$ define
\begin{equation}\label{mthetatrans(2,1)}\tilde f(g:\Lqxi)=\sum_{v\in X(\Lqxi)}f(gv)\end{equation}
Discussions as before, reduce the proof of theorem~\ref{(2,1)} to the following theorem.

\begin{thm}\label{unboundedequi(2,1)}
Let $G, H, K$ and $\{a_t\}$ be as in section~\ref{sec;inhomo} for the signature $(2,1)$ case. Let $Q_{\xi}$ be a quadratic form of signature $(2,1)$ as in the statement of theorem~\ref{(2,1)}. Let $\mathfrak{q}\in\SL_3(\bbr)$ and $\Lqxi$ be as above. Let $\nu$ be a continuous function on $K.$ Then we have
\begin{equation}\label{e:unboundedequi(2,1)}\limsup_{t\rightarrow\infty}\int_K\tilde{f}(a_tk:\Lqxi)\nu(k)dk\leq\int_{G/\Gamma}\hat{f}(g)d\mu(g)\int_K\nu(k)dk\end{equation}
where $\mu$ is the $H\ltimes\bbr^3$-invariant probability measure on the closed orbit $H\ltimes\bbr^3\cdot\Lqxi$.
\end{thm}
The proof of this theorem is very similar to that of theorem~\ref{unboundedequi(2,2)}. We will use the same notations as the previous sections for the sake of simplicity. The main notational difference to bare in mind is that in previous sections $L$ would denote a $2$-dimensional subspace where in this section $L$ is a one dimensional (null) subspace.

As before we need to study the subsets of $K$ where the function $\tilde{f}$ is ``large". We start by recalling the following; There is $c>0$ such that for all large $t$ and small $0<\delta\ll1$ we have
\begin{equation}\label{thinsets(2,1)}\{k\in K\h:\h \tilde{f}(a_tk\Lqxi)>\frac{c}{\delta}\}\subset\{k\in K\h: \alpha_{1}(a_tk\Lqxi)>\frac{1}{\delta}\}\hh\cup \{k\in K\h:\alpha_{2}(a_tk\Lqxi)>\frac{1}{\delta}\}\end{equation}
We fix $t$ and $\delta$ as above. Let us make two important remarks before we continue. The second name author would like to thank A.~Eskin for conversations regarding this remark.

\begin{remark}\label{r;(2,1)}
\begin{itemize}
\item[(i)] Using reduction theory of the orthogonal group we see that: if $\alpha_2(a_tk\Lqxi)>\frac{1}{\delta},$ then actually $\alpha_{1}(a_tk\Lqxi)>\frac{1}{\delta}.$ Hence we only need to study the contribution coming from $\alpha_1.$
\item[(ii)] The fact that $Q$ is a rational form implies that there exists $\delta_0$ depending only on $Q$ such that if $0<\delta<\delta_0$ and $\alpha_{1}(a_tk\Lqxi)>\frac{1}{\delta},$ then there is a null subspace $L$ such that $\alpha_{1}(a_tk\Lqxi)=\frac{1}{d(a_tkL)}$. This follows because $Q(\bbz^3)=B(\Lambda)$ is a discrete set and $a_tk$ is in the isometry group of $B.$
\end{itemize}
\end{remark}
The second remark above is the main reason for our assumption, $Q$ is a rational form. It is also the main reason we can handle this case in hand despite the more complicated case where $Q$ is an arbitrary form.

Given these two remarks and arguing as in section~\ref{sec:(2,2)}, we actually need to consider the set
\begin{equation}\label{thinsetnull(2,1)}\afrak_t({\delta})=\bigcup_{L\in\mathcal{N}}\{k\in K\h:\h d(a_tkL)<\delta\h\&\h\exists\hh v\in\bbz^3\h\hh\mbox{s. t.}\h a_tk(L+v+\mathfrak{q}\xi)\cap B(r)\neq\emptyset\}\end{equation}
where $\mathcal{N}$ is the set of null subspaces of $Q$ which are not exceptional subspaces and $B(r)$ is the ball of radius $r$ in $\bbr^3.$ We will show that; There exists $\eta>0$ depending on the Diophantine properties of $\xi$ such that: at most $\delta^{\eta}$ of null subspaces can have nontrivial contribution to~(\ref{thinsetnull(2,1)}). This will finish the proof of theorem~\ref{unboundedequi(2,1)} as we will see.

Let us remark that; The null subspaces coincide with the orbit of $\left(\begin{array}{cc}x_{11} & 0\\ 0 & 0\end{array}\right)$ under the action of $\SL_2(\bbr)$ described in section~\ref{sec;inhomo}. With this identification and using a theorem of Borel and Harish-Chandra, see the discussion in appendix B below, the primitive null vectors are identified with the $\SL_2(\bbz)$ orbits of vectors $\left(\begin{array}{cc}\pm1 & 0\\ 0 & 0\end{array}\right)$ under the same action. In what follows we will restrict ourselves to the $+1$ case. These are described as follows; For any primitive vector $(m,n)$ the corresponding primitive null vector is given by $\left(\begin{array}{cc}m^2 & mn\\ mn & n^2\end{array}\right).$ For any such subspace, $M$ say, we let $v^M$ denote this vector which we will refer to as the standard basis for $M.$

For any null subspace $L$ of $Q$ the subspace $\pfrak L$ is a (rational) null subspace of $B,$ we let $v^L$ denote the standard basis for $\pfrak L.$ For any such $L$ let $\afrak_t^L(\delta)$ be the corresponding set in~(\ref{thinsetnull(2,1)}). We have the following

\begin{lem}\label{fewbadsubspaces(2,1)}
There exists $0<\eta_1\ll\eta_2<1$ depending on the Diophantine properties of $\xi$ such that if $0<\delta\ll1$ is given and $T\geq\delta^{-\eta_2},$ then the number of null subspaces $L$ with $T/2\leq\|v^L\|\leq T$ such that $\afrak_t^L(\delta)\neq\emptyset$ is $O(\delta^{\eta_1}T).$ Furthermore there are at most two subspace $L$ with $\|v^L\|\leq\delta^{-\eta_2}$ for which
$\{\langle v^L\hh,\pfrak\xi\rb\}\leq O(\delta).$
The implied constants in the $O$-notation depend on $Q$ and $\xi.$
\end{lem}

\begin{proof}
First note that with the description of null subspaces of $B,$ which was given above, we see that the number of primitive vectors $w$ with $T/2\leq\|w\|\leq T$ and $B(w)=0$ is $O(T),$ see appendix B below for a discussion of similar statements. Let now $L$ be a null subspace with $\afrak_t^L(\delta)\neq\emptyset$ and suppose $v^L$ corresponds to the primitive vector $(m,n).$ We have
\begin{equation}\label{innerp(2,1)}\{\langle v^L\hh,\pfrak\xi\rb\} = n^2(\pfrak\xi)_1-2mn(\pfrak\xi)_2+m^2(\pfrak\xi)_3\leq c\delta\end{equation}
where $c$ is an absolute constant depending on $Q$ and $r.$ Since $\xi$ is Diophantine lemma~\ref{lineardiophantine} implies that $\pfrak\xi$ is Diophantine as well. 

Note that no three $v^L$ lie on one plane in $\bbr^3.$ Hence using lemma~\ref{lineardiophantine} we see that there exists $\eta_2$ depending on Diophantine properties of $\pfrak\xi$ such that~(\ref{innerp(2,1)}) cannot holds for more than two $L$'s with $\|v^L\|\leq\delta^{-\eta_2}.$

By virtue of Weyl's exponential sum estimates, see theorem 2.9, proposition 4.3 and theorem 8.6 in~\cite{GT}, we see that there are absolute constants $c_0,c_1\geq1$ such that for all small enough $\delta$ and all $0<\eta_1<1$; If~(\ref{innerp(2,1)}) holds for more that $\delta^{\eta_1}T$ subspaces $L$ with $T/2\leq\|v^L\|\leq T,$ then there exists $k\in\bbz$ with $|k|<\delta^{-c_1\eta_1}$ such that
\begin{equation}\label{diophantine(2,1)}
\max_i\{k(\pfrak\xi)_i\}\leq c_0\delta^{-c_1\eta_1}/T
\end{equation}

Now let $\eta_2$ be chosen as above and choose $\eta_1$ such that $\eta_2>(c_1+\kappa)\eta_1,$ where $\kappa$ is the Diophantine exponent of $\pfrak\xi$ and $c_1$ is the constant appearing in~(\ref{diophantine(2,1)}). With these choices we see that if $T>\delta^{-\eta_2}$ then~(\ref{diophantine(2,1)}) contradicts Diophantine assumption on $\pfrak\xi.$ Hence~(\ref{innerp(2,1)}) can hold for at most $\delta^{\eta_1}T$ subspaces as we wanted to show.
\end{proof}

Before starting the proof of theorem~\ref{unboundedequi(2,1)} we need to recall the following definition
\begin{equation}\label{e:intervals(2,1)}I_t^{L}(\delta)=\{k\in K\h:\h d(a_tkL)<\delta\}\end{equation}
Some important properties of these intervals were proved in~\cite{EMM2} and recalled in lemma~\ref{intervals}. What is important for us in this section is the property (ii) in lemma~\ref{intervals}. This property, tailored to our current assumptions, gives
\begin{equation}\label{intervals(2,1)}|\afrak_t^L(\delta)|\leq|I_t^L(\delta)|\approx \left(\frac{e^{-t}\delta}{T}\right)^{1/2}\end{equation}

\textit{Proof of theorem~\ref{unboundedequi(2,1)}.} Let $M$ be a large number which is fixed for now. Let $t>0$ be a large number. Assume $j$ is a large number also. We assume through out that $\eta=\eta_1$ where $\eta_1$ is as in lemma~\ref{fewbadsubspaces(2,1)}.
Let $L_0,L_1$ be nonexceptional null subspaces such that $\{\langle v^{L_k},\hh \pfrak\xi\rb\}<1/2^{j}$ for $k=0,1,$ and that $\|v^{L_0}\|\leq\|v^{L_1}\|$ are minimal with these properties, indeed we fix any two such subspaces if there are more than two subspaces satisfying these conditions. Assume that $e^{t}/2^{j+i_0+1}\leq\|v^L\|<e^{t}/2^{j+i_0}.$ We need to consider two cases

\textit{Case 1.} Assume that $\|v^{L_0}\|>2^{j\eta_2},$ where $\eta_2$ is as in lemma~\ref{fewbadsubspaces(2,1)}.

Now using lemma~\ref{fewbadsubspaces(2,1)} we see that for all $i\leq i_0$ the number of null subspaces $L,$ with $e^{t}/2^{j+i+1}\leq\|v^L\|<e^{t}/2^{j+i}$ such that $\{\langle v^L,\hh \pfrak\xi\rb\}<1/2^{j},$ is bounded by $e^t/2^{(1+\eta)j+i}.$ Also by our assumption in case 1 for all $i>i_0$ the subspaces of norm at most $e^t/2^{i+j}$ have no contribution to the set $\afrak_t(1/2^j)$ i.e. $\afrak_t^L(1/2^j)=\emptyset$ for all such subspaces $L$.

Hence if we use this fact and~(\ref{intervals(2,1)}) we have: If $j$ is such that case 1 holds then we have
\begin{equation}\label{genericcase(2,1)}|\bigcup_{L\hh\mbox{null}}\afrak_t^L(\frac{1}{2^j})|\leq \sum_i\sum_{\frac{e^t}{2^{i+j+1}}\leq\|v^L\|<\frac{e^t}{2^{i+j}}}|\afrak_t^L(\frac{1}{2^j})|\leq\frac{1}{2^{j(1+\eta)}}\sum_i {2^{-i/2}}\leq C\frac{1}{2^{j(1+\eta)}}\end{equation}

\textit{Case 2.} Assume $\|v^{L_0}\|\leq\frac{1}{2^{j\eta_2}},$ where $\eta_2$ is as in case 1. Note that there are at most two subspaces with this property as was shown in lemma~\ref{fewbadsubspaces(2,1)}. We need to consider the following possibilities

(i) $M\leq e^t/2^{i_0+j}\leq2^{j\eta_2}.$ We then have
\begin{equation}\label{exceptionalsubspace1(2,1)}|\afrak_t^{L_k}(\frac{1}{2^j})|\leq Ce^{-t}(e^{-t}\frac{e^t}{2^{i_0+j}})^{-1/2}\frac{1}{2^{j/2}}\leq C\frac{e^{-t/2}}{\sqrt M2^{j/2}},\h\h\h k=0,1\end{equation}
Now if we argue just as in case 1, we get
\begin{equation}\label{nongenericcase1(2,1)}|\bigcup_{L\hh\mbox{null}}\afrak_t^L(\frac{1}{2^j})|\leq \sum_i\sum_{\frac{e^t}{2^{i+j+1}}\leq\|v^L\|<\frac{e^t}{2^{i+j}}}|\afrak_t^L(\frac{1}{2^j})|\leq C(\frac{e^{-t/2}}{\sqrt M2^{j/2}}+\frac{1}{2^{j(1+\eta)}})\end{equation}

(ii) $e^t/2^{i_0+j}\leq2^{j\eta_2}\leq M.$ Note that we have $i_k=i_k(j),$ where $k=0,1$ and for each $j$ there are at most two such $i_k$'s. We will denote by $L(j,i_k)$ the subspace corresponding to $i_k=i_k(j),$ for $k=0,1.$ We have
\begin{equation}\label{nongenericcase2}|\afrak_t^{L(j,i_k)}(\frac{1}{2^j})|\leq Ce^{-t}2^{\frac{i_k(j)}{2}}\leq Ce^{-t/2},\hspace{4mm} k=0,1\end{equation}

(iii) $e^t/2^{i_0+j}\leq M\leq2^{j\eta_2}.$ Given $M,$ this can only hold for at most two subspace $L^0_M$ and $L^1_M$ furthermore there exists $\ell=\ell(M)$ such that if this holds for $j\geq\ell$ then these are exceptional subspaces which we have excluded already. Hence we assume this holds only for $j<\ell.$ We have
\begin{equation}\label{nongenericcase3(2,1)} |\afrak_t^{L^k_M}(\frac{1}{2^j})|\leq|I_t^{L_M}(1/2^j)|\leq e^{-t/2}2^{-j/2}\h\h\mbox{for}\h\h j<\ell\h\h\mbox{and}\h\h k=0,1\end{equation}

We collect these estimates in the following
\begin{itemize}
\item[(I)] $\sum_{s<j\leq t}2^j\frac{1}{2^{j(1+\eta)}}\leq \frac{1}{2^{s\eta}}$
\item[(II)] $ \sum_{s<j\leq t}2^j(\frac{e^{-t/2}}{\sqrt M2^{j/2}}+\frac{1}{2^{j(1+\eta)}})\leq \frac{1}{\sqrt M}+\frac{1}{2^{s\eta}}$
\item[(III)] $\sum_{s<j\leq\frac{\log M}{\eta}}2^je^{-t/2}\leq M^{\frac{1}{\eta}}e^{-t/2}$
\item[(IV)] $\sum_{s<j<\ell}2^j e^{-\frac{t}{2}}2^{-\frac{j}{2}}\leq e^{-\frac{t}{2}}2^{\frac{\ell}{2}}$
\end{itemize}

Fix $\epsilon>0$ arbitrarily small. Let $\eta$ be as before. Then we may choose $M=M(\epsilon)$ and $t=t(M,\epsilon)>0$ and $s=s(\eta,\epsilon)>0$ large enough such that
\begin{equation}\label{unboundedintegralnull(2,1)}\sum_{s<j\leq t}2^j|\bigcup_{L\hh\mbox{null}}\afrak_t^{L}(\frac{1}{2^j})|\leq C(\frac{1}{2^{s\eta}}+\frac{1}{\sqrt M}+M^{\frac{1}{\eta}}e^{-t/2}+e^{-\frac{t}{2}}2^{\frac{\ell}{2}})<\epsilon/4\end{equation}
The proof of the theorem~\ref{unboundedequi(2,1)} is now completed similar to that of theorems~\ref{unbddequinot(2,2)} and~\ref{unboundedequi(2,2)}. Indeed in order to control the unbounded part of the integral one uses remark~\ref{r;(2,1)} and reduces to the study of null subspaces. The required estimate for null subspaces then is provided by~(\ref{unboundedintegralnull(2,1)}) above.



\appendix
\section{Equidistribution of spherical averages}\label{sec:equi}
Let $G$ be a connected Lie group and $\Gamma$ a lattice in $G.$ we let $\pi$ denote the natural projection from $G$ onto $G/\Gamma.$ Let $H$ be a connected semisimple subgroup of $G$ and let $K$ be the maximal compact subgroup of $H$. The question which is addressed in this section is that of equidistribution of sets of the form $a_tKx,$
where $x\in G/\Gamma$ and $A=\{a_s\h:\h s\in\bbr\}$ is a suitable one parameter subgroup of $H.$ This is a well-studied question. The main tools, in the analysis, are indeed Ratner's theorem on classification of unipotent flow invariant measures on $G/\Gamma$ and linearization techniques for the action of unipotent groups on $G/\Gamma$ which were developed by Dani and Margulis..


Let $H$ and $W$ be closed subgroups of $G.$ Following Dani and Margulis~\cite{DM3} define $X(H,W)=\{g\in G\h:\h Wg\subset gH\}.$ We recall the following

\begin{thm}\label{danimargulis}\cite[Theorem 3]{DM3}
Let $G$ be a connected Lie group and $\Gamma$ a lattice in $G.$ Let $U=\{u_t\}$ be an $\mbox{\rm{Ad}}$-unipotent one parameter subgroup of $G.$ Let $\phi$ be a bounded continuous function on $G/\Gamma.$ Let $\mathcal{D}$ be a compact subset of $G/\Gamma$ and let $\vare>0$ be given. Then there exist finitely many proper closed subgroups $H_1=H_1(\phi,\mathcal{D},\vare),\cdots,H_k=H_k(\phi,\mathcal{D},\vare)$ such that $H_i\cap\Gamma$ is a lattice in $H_i$ for all $i,$ and compact subsets $C_1=C_1(\phi,\mathcal{D},\vare),\cdots,C_k=C_k(\phi,\mathcal{D},\vare)$ of $X(H_1,U),\cdots,X(H_k,U)$ respectively, for which the following holds: For any compact subset $F$ of $\mathcal{D}\setminus\bigcup_iCi\Gamma/\Gamma$ there exists $T_0>0$ such that for all $x\in F$ and $T>T_0$ we have
\begin{equation}\label{bequi}\left|\frac{1}{T}\int_0^T\phi(u_tx)dt-\int_{G/\Gamma}\phi dg\right|<\vare\end{equation}
where $dg$ is the Haar measure on $G/\Gamma.$
\end{thm}

We will apply the above theorem to a special case which we now describe. Let us fix a few notations to be used through out this section. Let $G=\SL_n(\bbr)\ltimes\bbr^n$ where as before we are considering the standard action of $\SL_n(\bbr)$ on $\bbr^n.$ Let $\vartheta$ denote the natural projection from $G$ onto $\SL_n(\bbr).$ Let $n=p+q$ where $p\geq2$ and $q\geq1.$ Let $H=\rm{SO}(p,q)$ and $K=\rm{SO}(p)\times\rm{SO}(q),$ and let $H^0$ denote the identity component of $H.$ We let $H_1\subset H$ be the subgroup that fixes $\{e_3,\cdots,e_{n-1}\}$ then $H_1\cong\rm{SO}(2,1).$ Hence there is a homomorphism $\rho$ with finite kernel from $\SL_2(\bbr)$ to $H_1\subset H^0$ such that $A=\{a_s\h:\h s\in\bbr\}=\rho(\rm{diag}(e^{s/2},e^{-s/2}))$ is a self-adjoint one-parameter subgroup of $\rm{SO}(2,1)$ and $U=\{u_t\h:\h t\in\bbr\}=\rho(\begin{array}{cc}1 & t\\ 0 & 1\end{array})$ is the corresponding expanding horospherical subgroup for $s>0.$

Note that the standard representation of $H^0$ on $\bbr^n$ is irreducible hence $H^0$ is a maximal connected subgroup of $H\ltimes\bbr^n.$ Note also that as $H^0$ is a maximal connected subgroup of $\SL_n(\bbr)$ we have that $H^0\ltimes\bbr^n$ is a maximal connected subgroup of $G.$

Let now $\Gamma$ be a lattice in $G.$ Then $\Gamma\cap\bbr^n$ is a lattice in $\bbr^n.$ We let $\Delta=\vartheta(\Gamma)=\Gamma/\Gamma\cap\bbr^n,$ this is a lattice in $\SL_n(\bbr).$ We will abuse the notation and let $\vartheta$ also denote the projection from $G/\Gamma$ onto $\SL_n(\bbr)/\Delta.$ We have

\begin{lem}
Let $x\in G/\Gamma$ then the orbit $H\vartheta(x)$ is closed in $\SL_n(\bbr)/\Delta$ if and only if $H\ltimes\bbr^n\hh x$ is closed in $G/\Gamma.$
\end{lem}

\begin{proof}
Note that closed orbits of $H\ltimes\bbr^n$  (resp. $H$) have a finite $H\ltimes\bbr^n$-invariant (resp. $H$-invariant) measure by section 3 in~\cite{Mar1}. Let $x=(g,v)\Gamma.$ Note that $H\ltimes\bbr^n\hh x=\vartheta^{-1}(H\vartheta(x))$ hence if the $H\vartheta(x)$ is closed so is $H\ltimes\bbr^n\hh x$. Suppose now that $\hrn\hh x$ is closed. Then $\Gamma_1=H\ltimes\bbr^n\cap(g,v)\Gamma(g,v)^{-1}$ is a lattice in $H\ltimes\bbr^n$ and hence $\Gamma_1\cap\bbr^n$ is a lattice in $\bbr^n$ and $\Gamma_1/\Gamma_1\cap\bbr^n$ is a lattice in $H.$ Hence $H\vartheta(x)$ is closed, as we wanted.
\end{proof}

The following is special case of theorem 4.4 in~\cite{EMM1}.

\begin{thm}\label{bequidistribution1}
Let the notation be as above. Further assume $\Lambda$ is a lattice in $H^0\ltimes\bbr^n.$ Let $\phi$ be a compactly supported continuous function on $H^0\ltimes\bbr^n/\Lambda.$ Then for every $\vare>0$ and any bounded measurable function $\nu$ on $K$ and every compact subset $\mathcal{D}$ of $H^0\ltimes\bbr^n/\Lambda$ there exist finitely many points $x_1,\cdots,x_\ell\in H^0\ltimes\bbr^n/\Lambda$ such that
\begin{itemize}
\item[(i)] the orbit $H^0\hh x_i$ is closed and has finite $H^0$-invariant measure, for all $i,$
\item[(ii)]  for any compact set $F\subset\mathcal{D}\setminus\bigcup_iH\hh x_i$ there exists $s_0>0$ such that for all $x\in F$ and $s>s_0$
\begin{equation}\label{bwaverage}\left|\int_K\phi(a_skx)\hh\nu(k)\hh dk-\int_{H^0\ltimes\bbr^n/\Lambda}\phi\hh dg\int_K\nu\hh dk\right|\leq\vare\end{equation}
\end{itemize}
\end{thm}

We will also need a slight variant of theorem 4.4 in~\cite{EMM1}. This is the content of the following

\begin{thm}\label{bequidistribution2}
Let $G,$ $H,$ $K$ and $\Gamma$ be as above. Let $A=\{a_s\h:\h s\in\bbr\}$ be as above also. Let $\phi$ be a compactly supported continuous function on $G/\Gamma.$ Then for every $\vare>0$ and any bounded measurable function $\nu$ on $K$ and every compact subset $\mathcal{D}$ of $G/\Gamma$ there exists finitely many points $x_1,\cdots,x_\ell\in G/\Gamma$ such that
\begin{itemize}
\item[(i)] the orbit $H\ltimes\bbr^n\hh x_i$ is closed and has finite $H\ltimes\bbr^n$-invariant measure, for all $i,$
\item[(ii)]  for any compact set $F\subset\mathcal{D}\setminus\bigcup_iH\ltimes\bbr^nx_i$ there exists $s_0>0$ such that for all $x\in F$ and $s>s_0$
\begin{equation}\label{bwaverage}\left|\int_K\phi(a_skx)\hh\nu(k)\hh dk-\int_{G/\Gamma}\phi\hh dg\int_K\nu\hh dk\right|\leq\vare\end{equation}
\end{itemize}
\end{thm}

\begin{proof}
The proof of this theorem goes along the same lines as in section 4 in~\cite{EMM1}. Let $U$ be as defined above. Let $H_i=H_i(\phi,K\mathcal{D},\vare)$ and $C_i=C_i(\phi,K\mathcal{D},\vare)$ for $i=1,\cdots,k$ be given as in theorem~\ref{danimargulis} corresponding to $U$. For $1\leq i\leq k$ define
\begin{equation}\label{kexceptinal}\mathcal{G}_i=\{g\in G\h:\h Kg\subset X(H_i,U)\}\end{equation}
Note that the group generated by $\bigcup_{k\in K}k^{-1}Uk$ is $H^0$ as $U$ is not contained in any proper normal subgroup of $H$ and $K$ is the maximal compact subgroup of $H.$ Let now $g\in \mathcal{G}_i$ then $k^{-1}Uk\subset gH_ig^{-1}$ for all $k\in K.$ Hence $H^0\subset gH_i^0g^{-1}.$ Now as $H^0$ acts irreducibly on $\bbr^n$ the only possibilities for $gH_i^0g^{-1}$ are $H^0$ and $H^0\ltimes\bbr^n.$
Thus we have
\begin{equation}\label{almostmaximal}gH_i^0g^{-1}\subset H^0\ltimes\bbr^n\hspace{4mm}{\rm{for \h any}}\h\hh g\in \mathcal{G}_i,\hspace{4mm}1\leq i\leq k\end{equation}
We also note that if $g\in G$ is such that $gH^0g^{-1}\subset H^0\ltimes\bbr^n$ then $g\in N_{\SL_n(\bbr)}(H^0)\ltimes\bbr^n.$ This fact and (\ref{almostmaximal}) say that if $g_1,g_2\in\mathcal{G}_i$ for some $i$ then $g_1^{-1}g_2\in N_{\SL_n(\bbr)}(H^0)\ltimes\bbr^n.$ Since $H^0\ltimes\bbr^n$ is of finite index in $N_{\SL_n(\bbr)}(H)\ltimes\bbr^n$ we get $\mathcal{G}_i$ can be covered by finitely many co sets of $H^0\ltimes\bbr^n.$ Hence there are finitely many points $x_1,\cdots,x_\ell\in G/\Gamma$ such that $H\ltimes\bbr^nx_i$'s are closed and have a finite $H\ltimes\bbr^n$-invariant measure for all $1\leq i\leq\ell$ and that
\begin{equation}\label{closedorbits}\bigcup_i\mathcal{G}_i\Gamma/\Gamma\subset\bigcup_{1\leq i\leq\ell}H\ltimes\bbr^nx_i\end{equation}
Note that since $X(H_i,U)$'s are analytic submanifolds of $G$ and $K$ is a connected we have that for any $g\in G\setminus\bigcup_i\mathcal{G}_i\Gamma/\Gamma$
\begin{equation}\label{singularmeasurezero1}|\{k\in K\h:\h kg\in\bigcup_iX(H_i,U)\Gamma\}|=0\end{equation}
Now (\ref{closedorbits}), (\ref{singularmeasurezero1}) and the fact that $C_i\subset X(H_i,U)$ give
\begin{equation}\label{singularmeasurezero2}|\{k\in K\h:\h kx\in\bigcup_iC_i\Gamma/\Gamma\}|=0\end{equation}
for any $x\in F\subset\mathcal{D}\setminus\bigcup_iH\ltimes\bbr^nx_i.$
Now if we apply lemma 4.2 in~\cite{EMM1} then there exists and open subset $W\subset G/\Gamma$ such that $\bigcup_iC_i\Gamma/\Gamma\subset W$ and $|\{k\in K\h:\h kx\in W\}|<\vare.$ Recall that $C_i=C_i(\phi,K\mathcal{D},\vare),$ hence there exists $T_0>0$ such that
\begin{equation}\label{unipotentaverage1}\left|\frac{1}{T}\int_0^T\phi(u_ty)-\int_{G/\Gamma}\phi dg\right|<\vare\end{equation}
for all $y\in K\mathcal{D}\setminus W.$ So we have
\begin{equation}\label{unipotentaverage2}\left|\frac{1}{T}\int_0^T\int_K\phi(u_tkx)-\int_{G/\Gamma}\phi dg\right|\int_K\nu dk\leq\sup_{k\in K}|\nu(k)|(1+2\sup_{y\in G/\Gamma}|\phi(y)|)\hh\vare\end{equation}
The rest of the argument is mutatis mutandis of the proof of theorem 4.4 in~\cite{EMM1} replacing theorem 4.3 in loc. cit. by what we proved above.
\end{proof}

\section{Number of quasinull subspaces}~\label{sec:counting}
This section is devoted to the proof of proposition~\ref{counting}. We will need results proved in~\cite[Section 10]{EMM2}. We will recall the statements in here for the convenience of the reader.

Consider the bilinear form in $6$-variables $Q^{(6)}(v,w)=v\wedge w.$ Then $V_1$ and $V_2$ in lemma~\ref{linearalgebra} are orthogonal with respect to this form and the restriction of $Q^{(6)}$ to $V_i$ has signature $(2,1).$ The following is a proved in the course of the proof of theorem 10.4 in~\cite{EMM2}.
\begin{prop}\label{rationalspace}{\rm{(cf.~\cite[Section 10]{EMM2})}}
Let $\tau>0$ be any sufficiently small number and let $T>2$. There exists a rational three dimensional subspace $U$ of $\bigwedge^2\bbr^4.$ With a reduced integral basis of norm at most $T^\tau$ whose projections into $V_2$ have norm less than $T^{\tau-1}$ such that one of the following holds
\begin{itemize}
\item[(a)] The restriction of $Q^{(6)}$ to $U$ is anisotropic over $\bbq.$ In which case (i) in theorem~\ref{approximation} holds.
\item[(b)] The restriction of $Q^{(6)}$ to $U$ splits over $\bbq.$ In which case (ii) in theorem~\ref{approximation} holds. Furthermore in this case $Q'$ as in loc. cit. is proportional to $f(U,U^\perp)$ where $U^\perp$ is the orthogonal complement of $U$ with respect to $Q^{(6)}$ and $f$ is as in lemma~\ref{linearalgebra}. Moreover the number of quasinull subspaces with norm between $T/2$ and $T$ which are not in $U$ or $U^\perp$ is $O(T^{1-\tau}).$
\end{itemize}
\end{prop}

Let us denote by $Q^{(3)}(v)$ the restriction of $Q^{(6)}$ to $U.$ As we mentioned before this is a form with signature $(2,1)$ and $U$ is a rational subspace. If $L$ is a quasinull subspace with $T/2\leq\|v^L\|\leq T$ then we have $Q^{(6)}(v^L)=0.$ Hence if $L$ is a quasinull subspace in $U$ then $Q^{(3)}(v^L)=0.$ In other words the proposition~\ref{counting} will follow from the following

\begin{prop}\label{counting2}
Let $Q_0(x,y,z)=2xz-y^2$ be the standard quadratic form of signature $(2,1)$ on $\bbr^3.$ Let  $\Delta=g\bbz^3$ where $g\in{\rm{GL}}_3(\bbq).$ Let $T>0$ a large parameter and let $\tau$ be a sufficiently small parameter. Assume that the entries of $g$ are rational numbers whose nominator and denominators are bounded by a fixed power of $T^\tau$ and also suppose that $|{\rm{det}}\hh g|\leq T^\tau.$ Further assume that the length of the shortest vector in $\Delta$ is $c=O(1).$ Then if $T$ is large enough we have
\begin{equation}\label{e:counting2}\#\{w\in\mathcal{P}(\Delta)\h:\h Q_0(w)=0\h\h\mbox{and}\h\h \|w\|\leq T\}<O(T)\end{equation}
where $\mathcal{P}(\Delta)$ denote the set of primitive vectors in $\Delta.$ \end{prop}

\begin{proof}
Let $H\subset\SL_3(\bbr)$ be the subgroup which preserves the form $Q_0$. Note that $H$ is a $\bbq$-group and that $H\cong\mbox{SO}(2,1)$ as $\bbq$-groups. Let $\Gamma=\mbox{SO}(Q_0)(\bbz).$ Let $\mathcal{C}=\{w\in \bbr^3\h:\h Q_0(w)=0\}$ be the light cone of $Q_0$ and let $\mathcal{C}_R=\{w\in\mathcal{C}\h:\h R/2\leq\|w\|\leq R\},$ for any $R>0.$ Note that $\mathcal{C}$ is dilation invariant. Let $\lambda=\mbox{disc}(\Delta)^{1/3}$ and define $\Delta_1=\frac{1}{\lambda}\Delta.$ The lattice $\Delta_1$ is unimodular and the length of the shortest vector in $\Delta_1$ is at least $\frac{c}{\lambda}$. Let $\bar{g}\in\mbox{PGL}_3(\bbq)$ denote the image of $g.$ Indeed $\Delta_1=\bar g\bbz^3.$ The counting problem in~(\ref{e:counting2}) will follow if we show that
\begin{equation}\label{e:counting3}\#\{w\in\mathcal{P}(\Delta_1)\h:\h Q_0(w)=0\h\h\mbox{and}\h\h \|w\|\leq \frac{T}{\lambda}\}<O(T)\end{equation}
Let $\Gamma_g=H\cap g\mbox{SL}_3(\bbz)g^{-1}.$ Since $g$ is rational matrix $\Gamma_g$ is a lattice in $H.$ Further we have that the form $Q_g=gQ_0g^t$ is a split form over $\bbq$ and hence $\Gamma_g$ is a nonuniform lattice. Let $\Gamma_m$ be a maximal lattice such that $\Gamma_g\subset \Gamma_m.$ There are only finitely many classes of maximal nonuniform arithmetic lattices in $H.$ Hence we will assume $\Gamma_m$ is in the same class as $\Gamma$ in the rest of the argument. Thus $\Gamma_m=h_m\Gamma h_m^{-1}$ where $h_m\in\overline{H}(\bbq)$ and $\overline{H}$ is the adjoint form of $H.$ Now using reduction theory of the orthogonal group there is a universal constant $C$ such that we can find $v\in \mathcal{P}(h_m\bbz^3)\cap\mathcal{C}$ with $\|v\|<C.$ We also have the shortest vector in $h_m\bbz^3$ is at least $T^{-k\tau}$ where $k$ is a fixed number.

Since both $\bar g$ and $h_m$ are in $\mbox{PGL}_3(\bbq)$ we can find a primitive vector $v_0\in\mathcal{P}(\Delta_1)$ which is a scalar multiple of $v.$ Note that $\|v_0\|\geq\frac{c}{\lambda}.$ Let $P=\mbox{Stab}(\bbr v)$ be a maximal $\bbq$-parabolic subgroup of $H.$ Let $P=AN$ be the Levi decomposition of $P.$ We have $N\cdot v=v$ and $A$ acts on $v$ by a character. By a theorem of Borel and Harish-Chandra there is a finite set $\Xi\subset H(\bbq)$ such that $H(\bbq)=\Gamma_m \Xi P(\bbq).$ Thus the $\Gamma_m$-orbits are characterized by the values of the character on $A(\bbq)$. By Witt's theorem we have that $H(\bbq)$ acts transitively on the rational points on the cone $\mathcal{C}.$ Hence we see that $\mathcal{C}\cap\mathcal{P}(\Delta_1)\subset\Gamma_m\Xi v_0.$ As $\Xi$ is a finite set we only need to consider $\Gamma_m v_0.$

Let $e_1=(1,0,0)$ and let $N_0\subset H$ be the stabilizer of $e_1.$ Let $B_R=\{h\in H\h:\h h\cdot e_1\in\mathcal{C}_R\}$ and let $\mathcal{B}_R=B_R/N_0.$
Let now $\chi_R$ denote the characteristic function of $\mathcal{C}_R$ Define the following
\begin{equation}\label{thetasum}F_R(h)=\sum_{\Gamma_m/\Gamma_m\cap N}\chi_R(h\gamma v)\hspace{3mm}\mbox{for}\h\h h\in H\end{equation}

This is a function on $H/\Gamma.$ Now using the well-developed machinery for counting problems using mixing property, see in particular~\cite{BO} and~\cite{MS}, there exists $\vare>0$ depending only on the spectral gap of $H$ and $\tau$ such that
\begin{equation}\label{e:counting4}F_R(e)=\frac{|N_0/N_0\cap\Gamma|}{|H/\Gamma|}\mbox{vol}(\mathcal{B}_R)(1+O(\mbox{vol}(\mathcal{B}_R)^{-\vare}))\end{equation}
where $\mbox{vol}$ denotes the $H$ invariant Haar measure on $H/N.$ Hence we have
\begin{equation}\label{e:counting}\#(\Gamma_m v_0\cap \mathcal{C}_R)\leq F_{R/\|v_0\|}(e)=O(R/\|v_0\|)\end{equation}
Since $\|v_0\|\geq\frac{c}{\lambda},$ this finishes the proof.
\end{proof}


\end{document}